\documentclass
[
    a4paper,
    DIV=14,
    abstract=true,
]
{scrartcl}

\usepackage
{
    amsmath,
    amssymb,
    amsthm,
    array,
    authblk,
    dsfont, 
    enumitem,
    graphicx,
    mathtools,
    nicefrac,
    tabularx,
    tcolorbox,
    tikz,
    xcolor,
}
\usepackage[normalem]{ulem}

\usepackage[scr=boondoxupr]{mathalfa}

\usepackage[utf8]{inputenc}

\usepackage[pdffitwindow=false,
            plainpages=false,
            pdfpagelabels=true,
            pdfpagemode=UseOutlines,
            pdfpagelayout=SinglePage,
            bookmarks=false,
            colorlinks=true,
            hyperfootnotes=false,
            linkcolor=blue,
            urlcolor=blue!30!black,
            citecolor=green!50!black]{hyperref}

\usepackage[bf,normal]{caption}
\usetikzlibrary{arrows, arrows.meta, shapes, fit, positioning}

\tcbuselibrary{skins}
\tcbset{width=\linewidth, enhanced, sidebyside, sidebyside align=center, halign upper=flush center, halign lower=flush center, boxrule=1pt, size=title, segmentation style={solid, line width=1pt}, colback=red!3,colframe=red!50}

\graphicspath{{pics/}}

\DeclareMathAlphabet{\mathpzc}{OT1}{pzc}{m}{it}

\newcommand{\subfiguretitle}[1]{{\scriptsize{#1}} \\}
\newcommand{\R}{\mathbb{R}}                                      
\newcommand{\pd}[2]{\frac{\partial#1}{\partial#2}}               
\newcommand{\innerprod}[2]{\left\langle #1,\, #2 \right\rangle}  
\newcommand{\ts}{\hspace*{0.1em}}                                
\newcommand{\mc}[2][]{\mathpzc{#2}{\smash[t]{\mathstrut}}_{#1}}  
\providecommand{\abs}[1]{\left\lvert #1 \right\rvert}            
\providecommand{\norm}[1]{\left\lVert #1 \right\rVert}           
\providecommand{\vdot}{\boldsymbol\cdot}                         

\newcommand\xqed[1]{\leavevmode\unskip\penalty9999 \hbox{}\nobreak\hfill \quad\hbox{#1}}
\newcommand{\exampleSymbol}{\xqed{$\triangle$}}

\DeclareMathOperator{\diag}{diag}

\DeclareMathOperator{\mspan}{span}

\newtheorem{theorem}{Theorem}[section]
\newtheorem{corollary}[theorem]{Corollary}
\newtheorem{lemma}[theorem]{Lemma}
\newtheorem{proposition}[theorem]{Proposition}
\newtheorem{definition}[theorem]{Definition}
\theoremstyle{definition}
\newtheorem{example}[theorem]{Example}
\newtheorem{remark}[theorem]{Remark}

\makeatletter
\renewcommand*\env@matrix[1][*\c@MaxMatrixCols c]{%
  \hskip -\arraycolsep
  \let\@ifnextchar\new@ifnextchar
  \array{#1}}
\makeatother

\mathtoolsset{centercolon} 

\allowdisplaybreaks

\DeclareCaptionLabelFormat{period}{#1~#2}
\makeatletter
\def\blfootnote{\gdef\@thefnmark{}\@footnotetext}
\makeatother

\begin{document}

\title{Dynamical systems and complex networks: \\ A Koopman operator perspective}
\author[1]{Stefan Klus\thanks{Corresponding author: \href{mailto:s.klus@hw.ac.uk}{s.klus@hw.ac.uk}}}
\author[2]{Nata\v sa Djurdjevac Conrad}
\affil[1]{School of Mathematical \& Computer Sciences, Heriot--Watt University, Edinburgh, UK}
\affil[2]{Zuse Institute Berlin, Berlin, Germany}

\date{}

\maketitle

\vspace*{-10ex}

\begin{abstract}
The Koopman operator has entered and transformed many research areas over the last years. Although the underlying concept---representing highly nonlinear dynamical systems by infinite-dimensional linear operators---has been known for a long time, the availability of large data sets and efficient machine learning algorithms for estimating the Koopman operator from data make this framework extremely powerful and popular. Koopman operator theory allows us to gain insights into the characteristic global properties of a system without requiring detailed mathematical models. We will show how these methods can also be used to analyze complex networks and highlight relationships between Koopman operators and graph Laplacians.
\end{abstract}

\section{Introduction}

This perspective article is meant to be a self-contained introduction to and review of transfer operators such as the Koopman operator and the Perron--Frobenius operator as well as an overview of different applications. We will first introduce the required foundations and then show how transfer operators can not only be used to analyze highly nonlinear dynamical systems but also complex networks. In particular, we will focus on relationships between transfer operators for continuous-time stochastic processes defined on a continuous state space---whether they be reversible, non-reversible but time-homogeneous, or time-inhomogeneous---and their discrete counterparts associated with random walks on undirected, directed, and time-evolving graphs.

Transfer operators play an important role in an increasing number of research fields. A few exemplary applications are illustrated in Figure~\ref{fig:transfer operator applications}. More details regarding the specific application areas can be found in the following publications:
(a)~molecular dynamics \cite{Sch99, NSVRW09, PWSKSHSCSF11, SS13, SP13, KBSS18},
(b)~fluid dynamics \cite{Schmid10, Mezic05, RMBSH09, BMM12, JSN14, KGPS18},
(c)~climate science \cite{FHRSSG12, FSvS14, GSZ15, ZG16, KHMN19, NTK21},
(d)~quantum physics~\cite{Pav14, KNH20, GKD+21, KNP22, Giannakis2022, LLASS22},
(e)~chaotic dynamical systems~\cite{Dellnitz97, DJ99, FD03, DKZ17, BGLR23},
(f)~system identification~\cite{MauGon16, BPK16, KNPNCS20, Kaiser21, ZZ23},
(g)~control theory~\cite{KorMez16, PeiKlul19, POR20, MMS20, BSH21, ZPSYX22}, and
(h)~graphs and networks~\cite{Djurdjevac12, RSH16, MFFS19, Sinha22, KD22, KT24}.
Koopman-based methods have also been applied to video data, EEG recordings, traffic flow data, stock prices, and various other data sets. The goal is always to identify characteristic dynamical properties of the underlying system from data.

\begin{figure}
    \centering
    \begin{tcolorbox}[width=0.49\linewidth, sidebyside=false, colback=white,colframe=black, nobeforeafter]
        \subfiguretitle{(a)}
        \includegraphics[width=\textwidth]{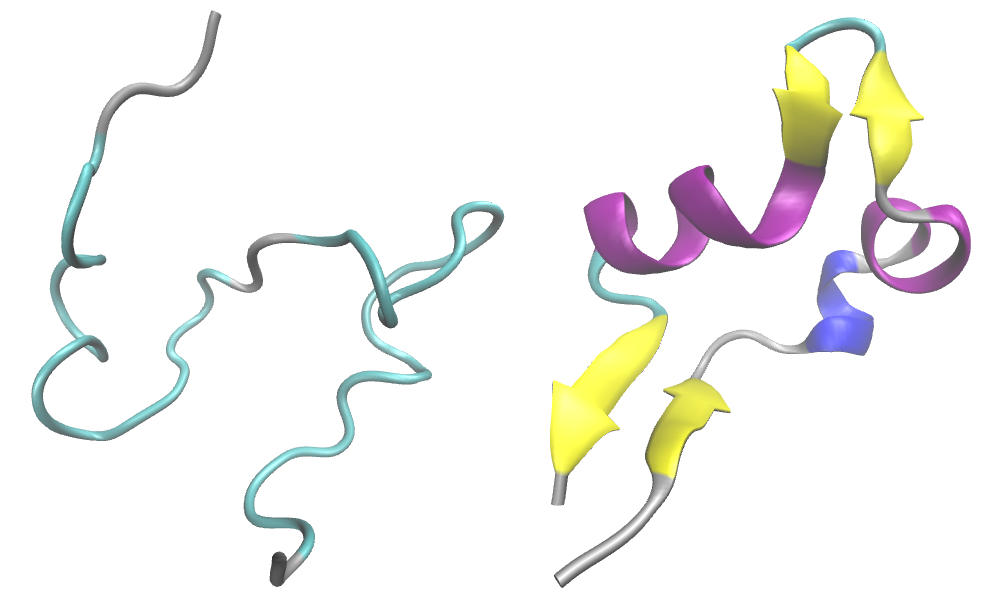}
    \end{tcolorbox} %
    \begin{tcolorbox}[width=0.49\linewidth, sidebyside=false, colback=white,colframe=black, nobeforeafter]
        \subfiguretitle{(b)}
        \includegraphics[width=\textwidth]{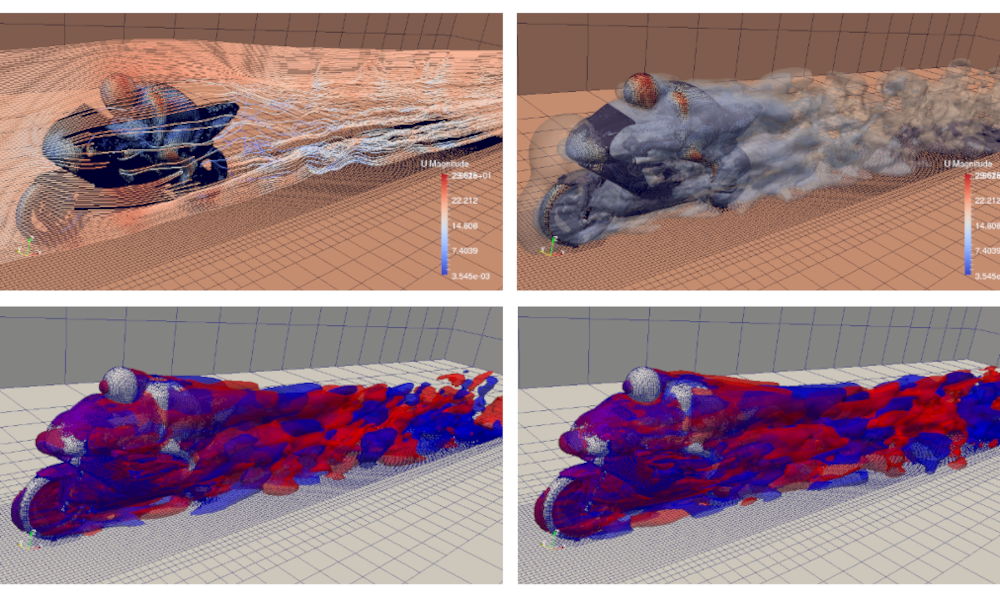}
    \end{tcolorbox} \\[0.5ex]
    \begin{tcolorbox}[width=0.49\linewidth, sidebyside=false, colback=white,colframe=black, nobeforeafter]
        \subfiguretitle{(c)}
        \includegraphics[width=\textwidth]{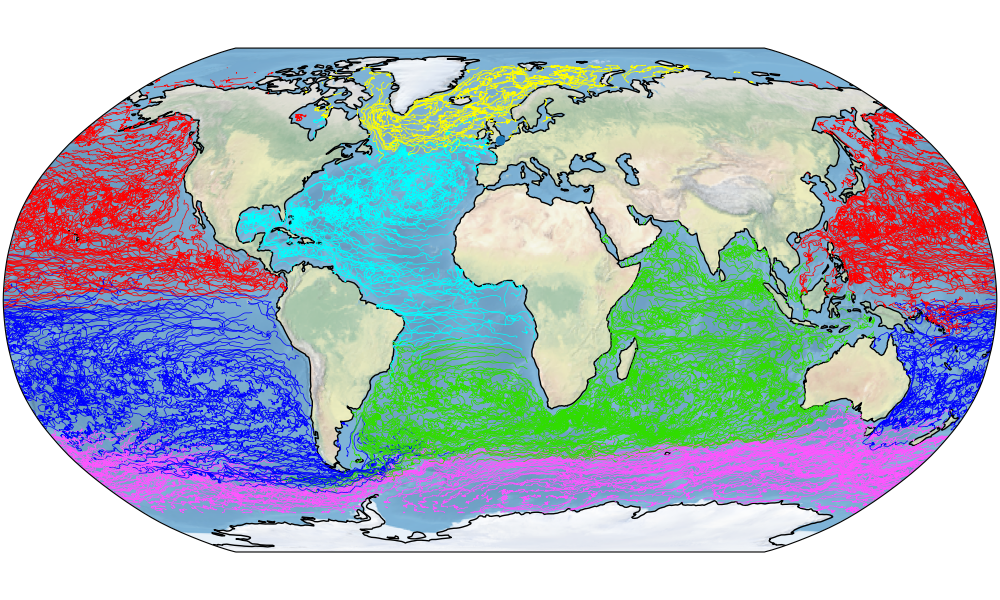}
    \end{tcolorbox}
    \begin{tcolorbox}[width=0.49\linewidth, sidebyside=false, colback=white,colframe=black, nobeforeafter]
        \subfiguretitle{(d)}
        \includegraphics[width=\textwidth]{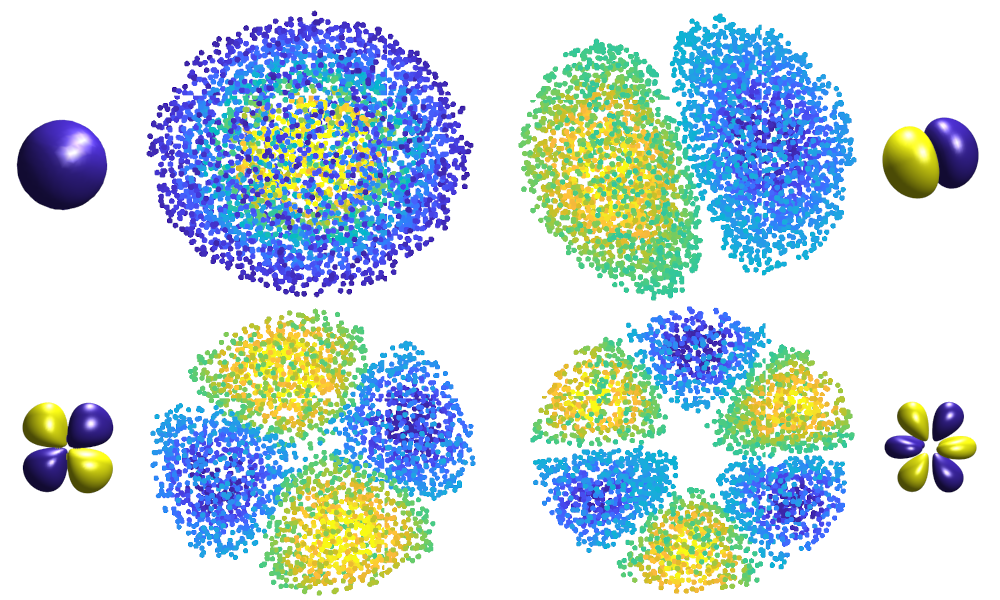}
    \end{tcolorbox} \\[0.5ex]
    \begin{tcolorbox}[width=0.49\linewidth, sidebyside=false, colback=white,colframe=black, nobeforeafter]
        \subfiguretitle{(e)}
        \includegraphics[width=\textwidth]{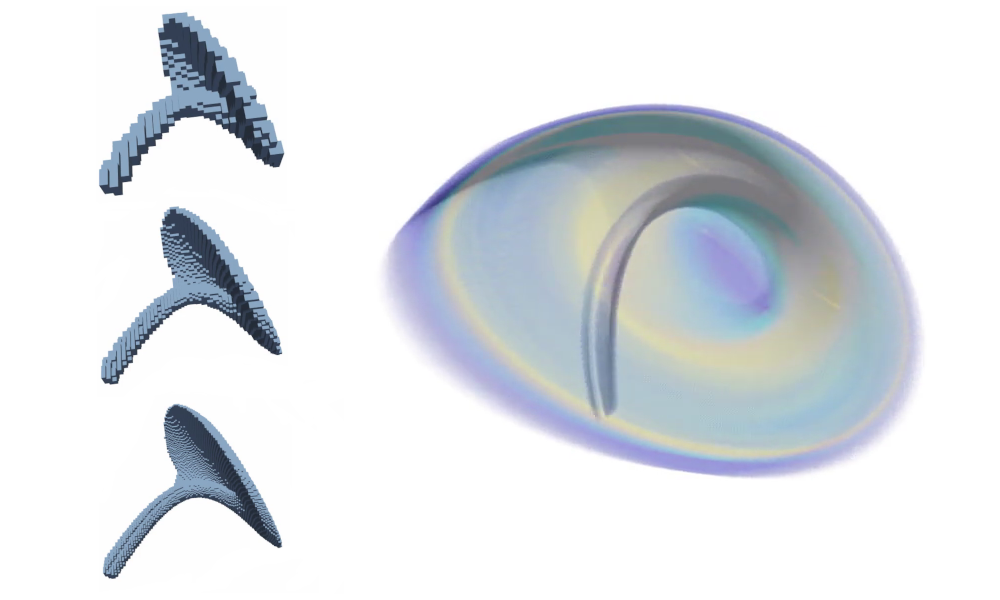}
    \end{tcolorbox}
    \begin{tcolorbox}[width=0.49\linewidth, sidebyside=false, colback=white,colframe=black, nobeforeafter]
        \subfiguretitle{(f)}
        \includegraphics[width=\textwidth]{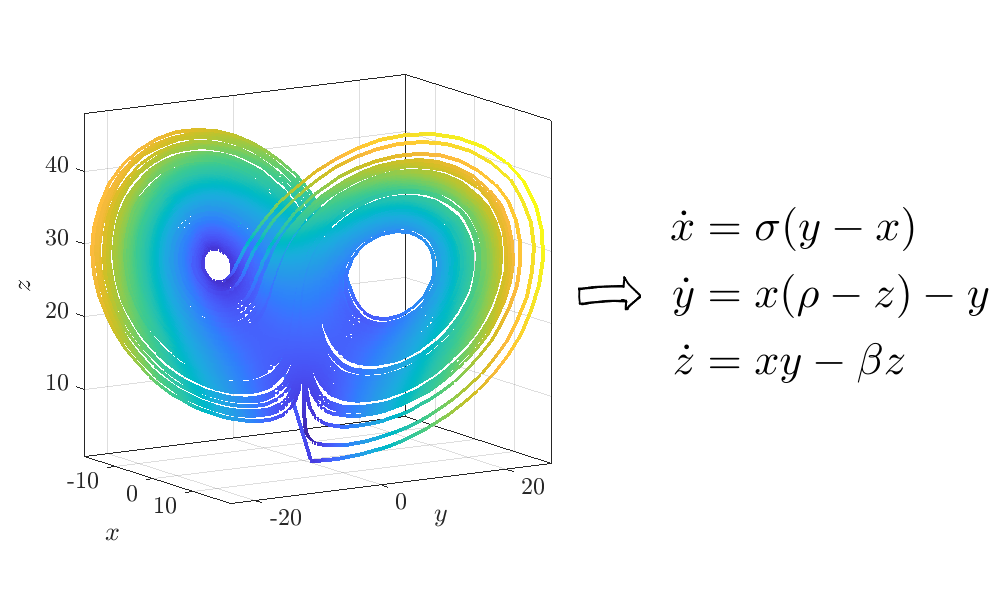}
    \end{tcolorbox} \\[0.5ex]
    \begin{tcolorbox}[width=0.49\linewidth, sidebyside=false, colback=white,colframe=black, nobeforeafter]
        \subfiguretitle{(g)}
        \includegraphics[width=\textwidth]{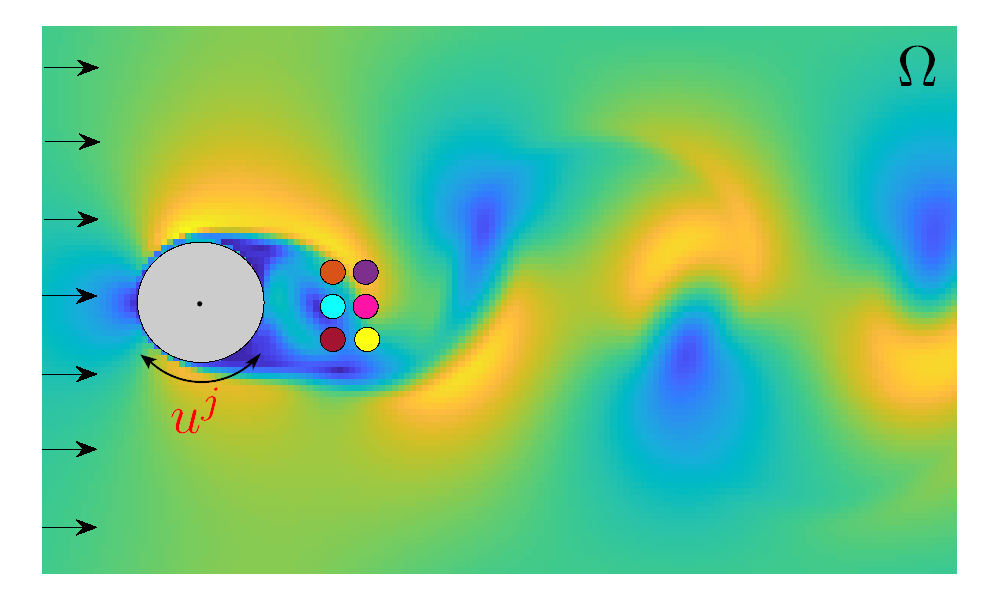}
    \end{tcolorbox}
    \begin{tcolorbox}[width=0.49\linewidth, sidebyside=false, colback=white,colframe=black, nobeforeafter]
        \subfiguretitle{(h)}
        \includegraphics[width=\textwidth]{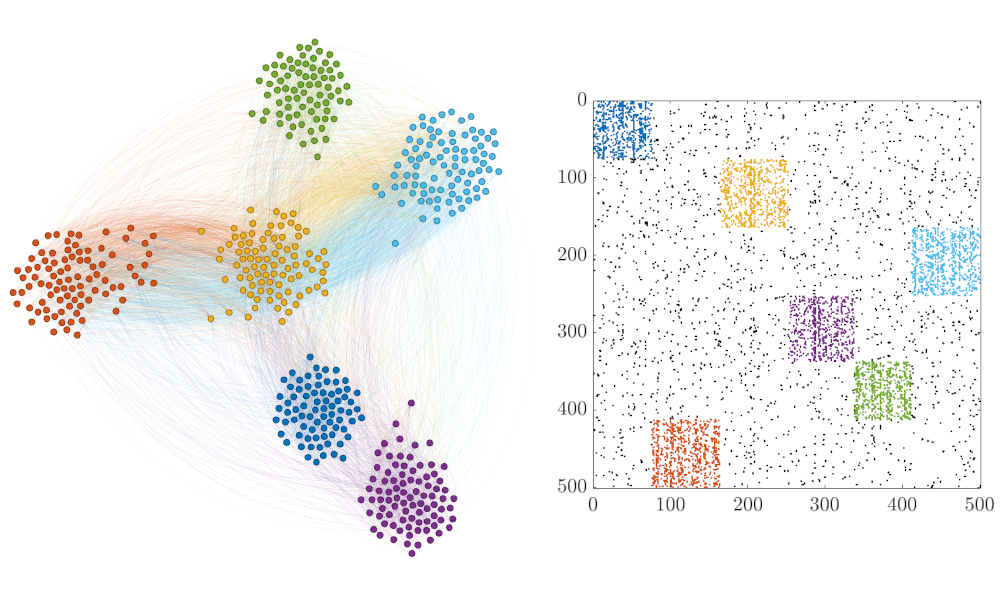}
    \end{tcolorbox}
    \caption{(a) Molecular dynamics: Folding properties of NTL9~\cite{KBSS18}. (b) Fluid dynamics: Decomposition into spatio-temporal patterns \cite{KGPS18}. (c)~Climate science: Detection of gyres in the oceans~\cite{KHMN19}. (d)~Quantum physics: Eigenstates of the hydrogen atom~\cite{KNH20}. (e)~Chaotic dynamical systems: Attractor of the Arneodo system~\cite{DKZ17}. (f)~System identification: Lorenz attractor \cite{BPK16}. (g)~Control theory: von Kármán vortex street~\cite{PeiKlul19}. (h)~Graphs and networks: Spectral clustering~\cite{KT24}.}
    \label{fig:transfer operator applications}
\end{figure}

In molecular dynamics, we are often interested in detecting slow processes such as the folding and unfolding of proteins \cite{Shaw11, SS13, PBB10, SKH23}. Conformations of molecules can be regarded as metastable states \cite{Davies82a, Davies82b, Bovier16}. Metastability implies that on short timescales the system appears to be equilibrated, but on longer timescales undergoes rare transitions between such metastable states. That is, the system will typically spend a long time in one part of the state space before it transitions to another region. Such metastable states are caused by deep wells in the energy landscape. Crossing the energy barrier from one well to another well is a rare event. Metastability is reflected in the spectrum of associated transfer operators, where the number of dominant eigenvalues corresponds to the number of metastable states. The eigenvalues are related to inherent timescales and the eigenfunctions contain information about the locations of the metastable states. Methods for detecting metastability also have important applications in climate science or agent-based modeling, see, for instance, \cite{Majda06, Serdukova16, Helfmann2021, WZShDjC21, NKS21, ZPIDjC23}.

In the same way, clusters in a graph can be interpreted as metastable states of a random walk process defined on the graph \cite{Luxburg07, Djurdjevac12}. A cluster is a set of vertices such that on short timescales the random walk process will, with a high probability, explore the cluster before moving to another part of the graph. Transitions between clusters are rare events and can be regarded as crossing an energy barrier caused by edges with low transition probabilities. Analogously, the cluster structure is reflected in the spectrum of associated graph Laplacians or, equivalently, transfer operators defined on graphs~\cite{KT24}. The number of dominant eigenvalues corresponds to the number of clusters and the associated eigenvectors contain information about the locations of the clusters. Over the last decades, many different clustering techniques based on properties of graph Laplacians have been developed, see, e.g., \cite{MS01, NgJorWei02, Verma03, Luxburg07, LaSch22}.

The decomposition into metastable states but also spectral clustering methods rely on the existence of a so-called \emph{spectral gap}. That is, we assume that there exist only a few isolated dominant eigenvalues close to one (or zero if we consider generators of transfer operators or graph Laplacians) and that the subsequent eigenvalues are significantly smaller. In practice, however, this is not necessarily the case. If there is no well-defined spectral gap, this typically indicates that the process is not metastable or that there are no weakly coupled clusters. Furthermore, the spectra of Koopman operators and graph Laplacians are in general only real-valued if the underlying stochastic process is reversible or, analogously, if the graph is undirected. Non-reversible processes and random walks on directed graphs typically result in complex-valued eigenvalues and eigenfunctions or eigenvectors, and conventional methods typically fail to identify meaningful slowly evolving spatio-temporal patterns or clusters. A generalization of metastable sets to non-reversible and time-inhomogeneous processes, called \emph{coherent sets} \cite{FrSaMo10, BaKo17, KWNS18, KHMN19}, is based on the forward--backward dynamics of the system and gives rise to self-adjoint operators---with respect to suitably weighted inner products---so that the spectrum becomes real-valued again. Coherent sets are regions of the state space that are only slowly dispersed by the dynamics and can be used to understand mixing properties of fluids. The notion of coherence has been extended to graphs in \cite{KD22, KT24}. The resulting spectral clustering methods can be used to identify clusters in directed and also time-evolving graphs.

The aim of this work is to illustrate how methods developed for the analysis of complex dynamical systems can be applied to graphs, but we will also turn the question around and ask how graph algorithms can be used to analyze dynamical systems in the hope that this will lead to cross-fertilization of the two scientific disciplines. The remainder of the article is structured as follows: We will first introduce transfer operators associated with stochastic dynamical systems and random walks on graphs in Section~\ref{sec:Stochastic processes, random walks, and transfer operators}, illustrate how these operators can be approximated and estimated from data in Section~\ref{sec:Approximation of transfer operators}, and then show that different types of dynamics induce different graph structures in Section~\ref{sec:Graph representations of stochastic processes}. Open problems and future work will be discussed in Section~\ref{sec:Conclusion}.

\section{Stochastic processes, random walks, and transfer operators}
\label{sec:Stochastic processes, random walks, and transfer operators}

We will briefly review the required dynamical systems and graph theory basics.

\subsection{Stochastic processes}

Given a stochastic process $ \{X_t\}_{t \ge 0} $ defined on a bounded state space $ \mathbb{X} \subset \R^d $ and any measurable set $ \mathbb{A} $, we assume there exists a transition density function $ p_{t,\tau} \colon \mathbb{X} \times \mathbb{X} \to \R_+ $ such that
\begin{equation*}
    \mathbb{P}\big[X_{t+\tau} \in \mathbb{A} \mid X_t = x\big] = \intop_\mathbb{A} p_{t,\tau}(x, y) \, \mathrm{d}y,
\end{equation*}
see \cite{KKS16} for a detailed derivation. That is, the function $ p_{t,\tau}(x, \cdot) $ describes the probability density of $ X_{t+\tau} $ given that $ X_t = x $. We will in particular consider processes governed by \emph{stochastic differential equations} (SDEs) of the form
\begin{equation} \label{eq:SDE}
    \mathrm{d}X_t = b(X_t, t) \ts \mathrm{d}t + \sigma(X_t, t) \ts \mathrm{d}W_t,
\end{equation}
where $ b \colon \R^d \times \R_+ \to \R^d $ is the drift term, $ \sigma \colon \R^d \times \R_+ \to \R^{d \times d} $ the diffusion term, and $ W_t $ a $ d $-dimensional Wiener process.

\begin{example} \label{ex:quadruple-well system}
The \emph{overdamped Langevin equation}
\begin{equation} \label{eq:overdamped Langevin}
    \mathrm{d}X_t = -\nabla V(X_t) \ts \mathrm{d}t + \sqrt{2 \beta^{-1}} \ts \mathrm{d}W_t
\end{equation}
plays an essential role in molecular dynamics. Here, $ V \colon \R^d \to \R $ is a given energy potential and $ \beta $ the inverse temperature. As a guiding example, we will consider the case $ d = 2 $ and define
\begin{equation*}
    V(x) = x_1^4 - \tfrac{1}{16} x_1^3 - 2 \ts x_1^2 + \tfrac{3}{16}x_1 + \tfrac{9}{8} + x_2^4 - \tfrac{1}{8} x_2^3 - 2 \ts x_2^2 + \tfrac{3}{8} x_2 + \tfrac{5}{4}
\end{equation*}
and set $ \beta = 3 $. This energy landscape has four wells and the deepest well $(-1,-1)$ is located in the lower left quadrant and the shallowest well $(1,1)$ in the upper right quadrant. The potential and trajectories generated using the Euler--Maruyama method with step size $ h = 10^{-3} $ are shown in Figure~\ref{fig:quadruple-well illustration}. \exampleSymbol

\begin{figure}
    \centering
    \begin{minipage}[t]{0.302\textwidth}
        \centering
        \subfiguretitle{(a)}
        \vspace*{1ex}
        \includegraphics[width=\textwidth]{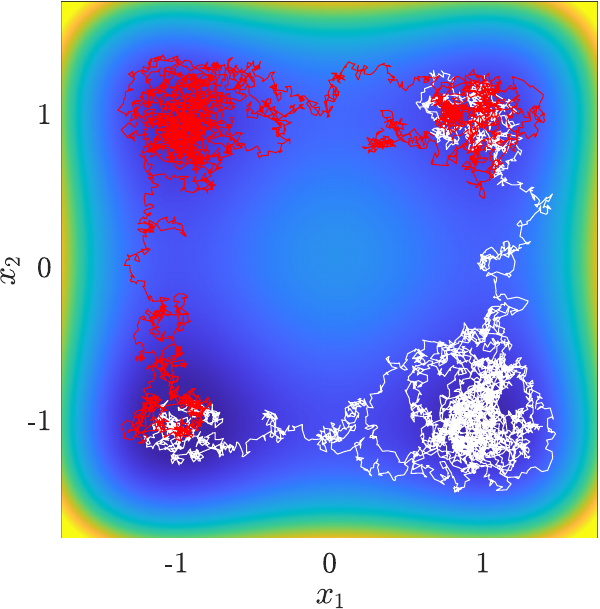}
    \end{minipage}
    \hspace*{4ex}
    \begin{minipage}[t]{0.37\textwidth}
        \centering
        \subfiguretitle{(b)}
        \vspace*{1ex}
        \includegraphics[width=\textwidth]{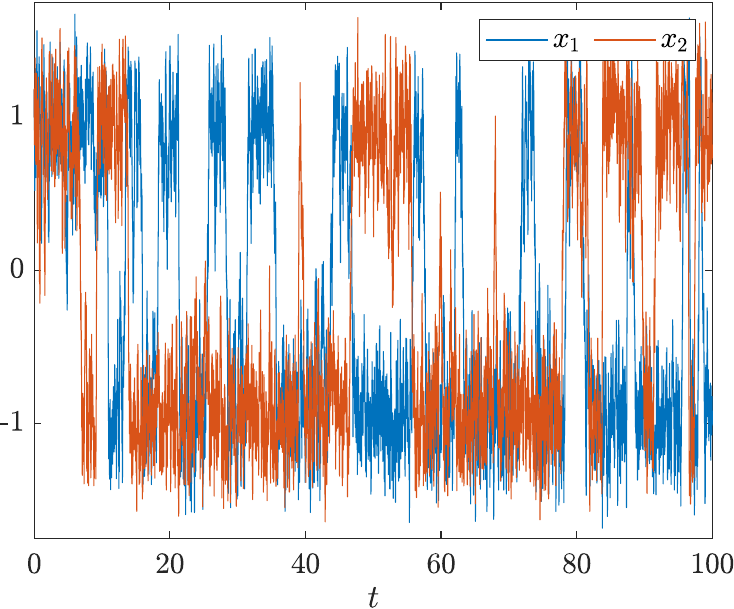}
    \end{minipage}
    \hspace*{2ex}
    \caption{(a) Quadruple-well potential and two short trajectories starting in the lower left and upper right wells. (b)~One long simulation of the stochastic differential equation. The system exhibits metastable behavior, that is, it will typically stay for a comparatively long time in one well before crossing the energy barrier and moving to another well.}
    \label{fig:quadruple-well illustration}
\end{figure}

\end{example}

Given a fixed lag time $ \tau $, considering the state only at multiples of~$ \tau $, i.e., $ X_0 $, $ X_{\tau} $, $ X_{2 \ts \tau} $, \dots, induces a discrete-time non-deterministic dynamical system $ \Theta_{t,\tau} $, given by $ \Theta_{t,\tau}(x) \sim p_{t,\tau}(x, \cdot) $. The lag time should be chosen in such a way that it is not too small (transitions of interest might not have occurred yet) and not too large (timescales of interest might be damped out and generating data will be computationally expensive).

A process is called \emph{time-homogeneous} if the transition probabilities from time $ t $ to time $ t + \tau $ depend only on the time difference $ \tau $ but not on $ t $ itself. This is, for instance, the case for the overdamped Langevin equation \eqref{eq:overdamped Langevin}. If the SDE is time-inhomogeneous, then also the transition densities and the transfer operators that we will introduce below are time-dependent. To simplify the notation, we will from now on omit this time-dependence and write $ p_\tau(x, y) $ instead of $ p_{t,\tau}(x, y) $, it will be clear from the context whether or not the system is time-homogeneous.

\subsection{Random walk processes on graphs}

A discrete-time stochastic process $ \{X_t\}_{t \in \mathbb{N}} $ on a discrete state space $\mathbb{X}$ is called a \emph{Markov chain} if it satisfies the \emph{Markov property}, which means that given a present state, the future and past jumps are independent of each other, i.e.,
\begin{equation*}
    \mathbb{P}[X_{t+1}=y \mid X_t = x, X_{t-1}=x_{t-1},\dots,X_0 =x_0] = \mathbb{P}[X_{t+1} = y \mid X_t = x],
\end{equation*}
with $ x_i, x, y \in \mathbb{X} $ and $ i = 0,...,{t-1} $. More rigorous definitions and detailed derivations can be found in \cite{Sarich2011}. For time-homogeneous Markov chains, i.e.,
\begin{equation*}
    \mathbb{P}[X_{t+1} = y \mid X_t = x] = \mathbb{P}[X_t = y \mid X_{t-1} = x] = \dots = \mathbb{P}[X_1 = y \mid X_0 = x],
\end{equation*}
we define a transition probability from state $ x $ to state $ y $ by
\begin{equation*}
    p(x,y) = \mathbb{P}[X_{t+1} = y \mid X_t = x]
\end{equation*}
for all $ x, y \in \mathbb{X} $. It holds that
\begin{equation*}
    p(x,y)\geq 0 \quad \text{and} \quad \sum_{y\in \mathbb{X}}p(x,y) = 1.
\end{equation*}
We will in particular consider Markov chains associated with graphs. Let $ \mc{G} = (\mc{X}, \mc{E}, w) $ be a \emph{weighted directed graph}, where $ \mc{X} = \{ \mc[1]{x}, \dots, \mc[\mc{n}]{x} \} $ is a set of $ \mc{n} $ vertices (also called nodes), $ \mc{E} \subseteq \mc{X} \times \mc{X} $ a set of edges (also called links), and $ w \colon \mc{X} \times \mc{X} \to \R_+ $ a \emph{weight function}. Here, $ w(\mc[i]{x}, \mc[j]{x}) > 0 $ is the weight of the edge $ (\mc[i]{x}, \mc[j]{x}) \in \mc{E} $ and $ w(\mc[i]{x}, \mc[j]{x}) = 0 $ if $ (\mc[i]{x}, \mc[j]{x}) \notin \mc{E} $. For a given graph $ \mc{G}$, we can then define a Markov chain as a \emph{random walk process} determined by the \emph{transition matrix} $ S \in \R^{\mc{n} \times \mc{n}} $, such that
\begin{equation} \label{eq:transition matrix}
    s_{ij} = p(\mc[i]{x}, \mc[j]{x}), \quad i, j = 1, \dots, \mc{n}.
\end{equation}
We will consider standard discrete-time random walks, where the transition probability of going from a vertex $ \mc[i]{x} $ to a vertex $ \mc[j]{x} $ is given by
\begin{equation} \label{eq:StRW}
    p(\mc[i]{x},\mc[j]{x}) = \frac{w(\mc[i]{x},\mc[j]{x})}{d(\mc[i]{x})}, \quad \text{with } d(\mc[i]{x}) = \sum_{\mc[j]{x} \in \mc{X}} w(\mc[i]{x},\mc[j]{x}).
\end{equation}
That is, $ d(\mc[i]{x}) $ is the weighted out-degree of vertex $ \mc[i]{x} $.

\begin{example} \label{ex:undirected graph}
In order to illustrate the close relationships between a continuous-time process defined on a continuous state space and random walks on a graph, we coarse-grain the quadruple-well problem introduced in Example~\ref{ex:quadruple-well system} by subdividing the domain $ \mathbb{X} = [-1.75, 1.75] \times [-1.75, 1.75] $ into $ 16 \times 16 $ equally sized boxes. Each box is represented by a vertex and the weight of an edge connecting two vertices is defined to be the transition probability between the corresponding boxes. This will be explained in more detail below. The random-walk process on the induced graph again exhibits metastable behavior as shown in Figure~\ref{fig:random walk illustration}. \exampleSymbol

\begin{figure}
    \centering
    \begin{minipage}[t]{0.30\textwidth}
        \centering
        \subfiguretitle{(a)}
        \vspace*{1ex}
        \includegraphics[width=\textwidth]{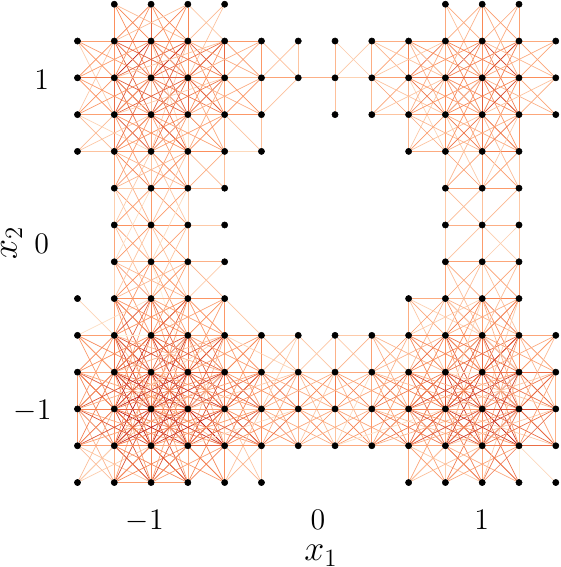}
    \end{minipage}
    \hspace*{4ex}
    \begin{minipage}[t]{0.37\textwidth}
        \centering
        \subfiguretitle{(b)}
        \vspace*{1ex}
        \includegraphics[width=\textwidth]{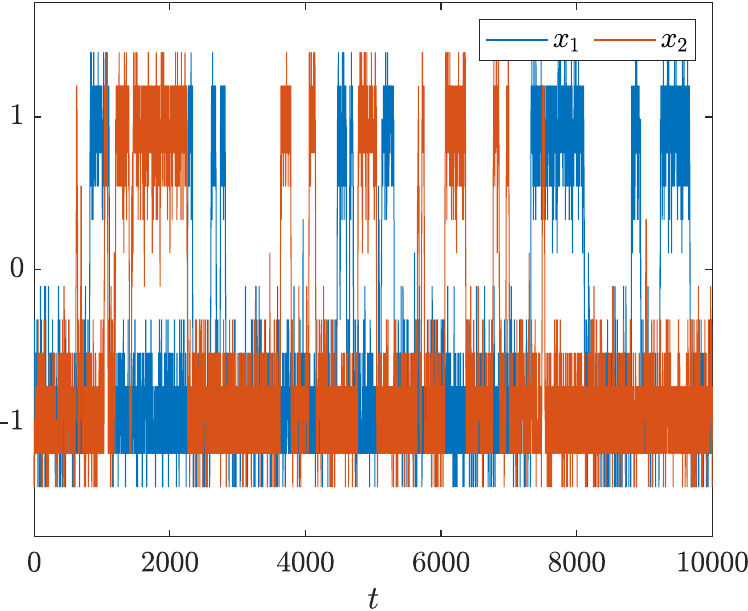}
    \end{minipage}
    \hspace*{2ex}
    \caption{(a) An undirected graph with four weakly coupled clusters based on the quadruple-well problem. Darker edges have larger edge weights. (b)~Random walk on the graph, where the locations of the nodes are used as coordinates. The process is metastable, i.e., a random walker will generally spend a long time in one cluster before moving to another one.}
    \label{fig:random walk illustration}
\end{figure}

\end{example}

This illustrates that metastable sets and clusters are closely related. The random walk process can be viewed as a stochastic differential equation discretized in time and space.

\subsection{Transfer operators}\label{sec:TransferOperators}

In what follows, we will often introduce two different variants of definitions, one for stochastic differential equations (shown on the left) and one for random walks on graphs (shown on the right).

\begin{definition}[Probability density]
A \emph{probability density} is a nonnegative function $ \mu $ satisfying
\begin{tcolorbox}[sidebyside align=bottom, after skip=10pt]
$
    \displaystyle \intop_\mathbb{X} \mu(x) \ts \mathrm{d}x = 1,
$
\tcblower
$
    \displaystyle \sum_{x \in \mc{X}} \mu(x) = 1.
$
\end{tcolorbox}
\end{definition}

Let $ L_\mu^q $ denote the $ \mu $-weighted $ L^q $-space of functions such that
\begin{tcolorbox}[sidebyside align=bottom]
$
    \displaystyle \norm{f}_{L_\mu^q} := \intop_\mathbb{X} \abs{f(x)}^q \mu(x) \ts \mathrm{d} x < \infty,
$
\tcblower
$
    \displaystyle \norm{f}_{L_\mu^q} := \sum_{x \in \mc{X}} \abs{f(x)}^q \mu(x) < \infty.
$
\end{tcolorbox}
Unweighted spaces will simply be denoted by $ L^q $. The $ \mu $-weighted duality pairing is defined by
\begin{tcolorbox}[sidebyside align=bottom]
$
    \displaystyle \innerprod{f}{g}_\mu = \intop_\mathbb{X} f(x) \ts g(x) \ts \mu(x) \ts \mathrm{d} x,
$
\tcblower
$
    \displaystyle \innerprod{f}{g}_\mu = \sum_{x \in \mc{X}} f(x) \ts g(x) \ts \mu(x).
$
\end{tcolorbox}

\begin{definition}[Transfer operators]
Let $ \rho $ be a probability density and $ f $ an observable.
Given a strictly positive initial density $ \mu $, assume that there exists $ u $ such that $ \rho = \mu \ts u $. We define the following \emph{transfer operators}:
\begin{tcolorbox}[sidebyside align=center]
\hspace*{-2ex}
\scalebox{0.91}{
$
\everymath={\displaystyle}
\arraycolsep=1.4pt\def\arraystretch{2}
\begin{array}{rl}
    \mathcal{P}^{\ts\tau} \rho(x) &= \!\intop_\mathbb{X}\! p_\tau(y, x) \ts \rho(y) \ts \mathrm{d}y, \\
    \mathcal{K}^{\ts\tau} f(x) &= \!\intop_\mathbb{X}\! p_\tau(x, y) \ts f(y) \ts \mathrm{d}y, \\
    \mathcal{T}^{\ts\tau} u(x) &= \frac{1}{\nu(x)} \!\intop_\mathbb{X}\! p_\tau(y, x) \ts \mu(y) \ts u(y) \ts \mathrm{d}y, \\
    \mathcal{F}^{\ts\tau} u(x) &= \!\intop_\mathbb{X}\! p_\tau(x, y) \frac{1}{\nu(y)} \!\intop_\mathbb{X}\! p_\tau(z, y) \ts \mu(z) \ts u(z) \ts \mathrm{d}z \ts \mathrm{d}y, \\
    \mathcal{B}^{\ts\tau} f(x) &= \frac{1}{\nu(x)} \!\intop_\mathbb{X}\! p_\tau(y, x) \ts \mu(y) \ts \!\intop_\mathbb{X}\! p_\tau(y, z) \ts f(z) \ts \mathrm{d}z \ts \mathrm{d}y,
\end{array}
$}
\tcblower
\scalebox{0.91}{
$
\everymath={\displaystyle}
\arraycolsep=1.4pt\def\arraystretch{2}
\begin{array}{rl}
    \mathcal{P} \rho(x) &= \sum_{y \in \mc{X}} p(y, x) \ts \rho(y), \vphantom{\intop_\mathbb{X}} \\
    \mathcal{K} f(x) &= \sum_{y \in \mc{X}} p(x, y) \ts f(y), \vphantom{\intop_\mathbb{X}} \\
    \mathcal{T} u(x) &= \frac{1}{\nu(x)} \sum_{y \in \mc{X}} p(y, x) \ts \mu(y) \ts u(y), \vphantom{\intop_\mathbb{X}} \\
    \mathcal{F} u(x) &= \sum_{y \in \mc{X}} p(x, y) \frac{1}{\nu(y)} \sum_{z \in \mc{X}} p(z, y) \ts \mu(z) \ts u(z), \vphantom{\intop_\mathbb{X}} \\
    \mathcal{B} f(x) &= \frac{1}{\nu(x)} \sum_{y \in \mc{X}} p(y, x) \ts \mu(y) \sum_{z \in \mc{X}} p(y, z) \ts f(z), \vphantom{\intop_\mathbb{X}}
\end{array}
$}
\end{tcolorbox}
\noindent where $ \nu = \mathcal{P} \mu $ is the resulting image density, i.e.,
\begin{tcolorbox}[sidebyside align=bottom, after skip=5pt]
\vspace*{1ex}
\scalebox{0.91}{
$
    \displaystyle \nu(x) = \intop_\mathbb{X} p_\tau(y, x) \ts \mu(y) \ts \mathrm{d}y,
$}
\tcblower
\scalebox{0.91}{
$
    \displaystyle \nu(x) = \sum_{y \in \mc{X}} p(y, x) \ts \mu(y),
$}
\end{tcolorbox}
\noindent also assumed to be strictly positive.
\end{definition}

Here, $ \mathcal{P} \colon L^2 \to L^2 $ is the \emph{Perron--Frobenius operator}, which describes the evolution of probability densities, and $ \mathcal{T} \colon L_\mu^2 \to L_\nu^2 $ a \emph{reweighted Perron--Frobenius operator}, which propagates probability densities w.r.t.\ the reference density $ \mu $. Further, $ \mathcal{K} \colon L^2 \to L^2 $ or $ \mathcal{K} \colon L_\nu^2 \to L_\mu^2 $, depending on whether we consider it to be the adjoint of $ \mathcal{P} $ or $ \mathcal{T} $, is the so-called \emph{stochastic Koopman operator}, which describes the evolution of observables.\!\footnote{Typically, $ \mathcal{P}^{\ts\tau} $ and $ \mathcal{K}^{\ts\tau} $ are first defined on $ L^1 $ and $ L^\infty $, respectively, but the operators can be extended to other function spaces $ L^q $ and $ L^{q\prime} $, with $ \frac{1}{q} + \frac{1}{q\prime} = 1 $, see \cite{BaRo95, KKS16} for more details.} The operators $ \mathcal{F} = \mathcal{K} \mathcal{T} \colon L_\mu^2 \to L_\mu^2 $ and $ \mathcal{B} = \mathcal{T} \mathcal{K} \colon L_\nu^2 \to L_\nu^2 $ are called \emph{forward--backward operator} and \emph{backward--forward operator}, respectively. Under mild conditions on the underlying system, see \cite{KWNS18}, the above operators are compact, which we assume to be the case in what follows. For a detailed introduction to transfer operators, see \cite{LaMa94, DJ99, Froyland13, KKS16, BaKo17}. The corresponding graph transfer operators on the right were introduced in \cite{KT24}. Since the state space is in this case finite-dimensional, all norms are equivalent. Note that the transfer operators on the left explicitly depend on the lag time~$ \tau $, whereas the operators on the right are associated with the discrete-time random walk process. When statements hold for the discrete case and the continuous case, we will sometimes omit the superscript~$ \tau $ to highlight properties these operators have in common.

\begin{example}
Let us illustrate the definitions of the transfer operators $ \mathcal{P}^\tau $ and $ \mathcal{K}^\tau $ using the quadruple-well problem introduced in Example~\ref{ex:quadruple-well system}. The Perron--Frobenius operator describes the evolution of probability densities and the Koopman operator the evolution of observables as shown in Figure~\ref{fig:transfer operator illustration}. \exampleSymbol

\begin{figure}
    \centering
    \begin{minipage}[t]{0.24\textwidth}
        \centering
        \subfiguretitle{(a) $ t = 0 $}
        \includegraphics[width=0.98\textwidth]{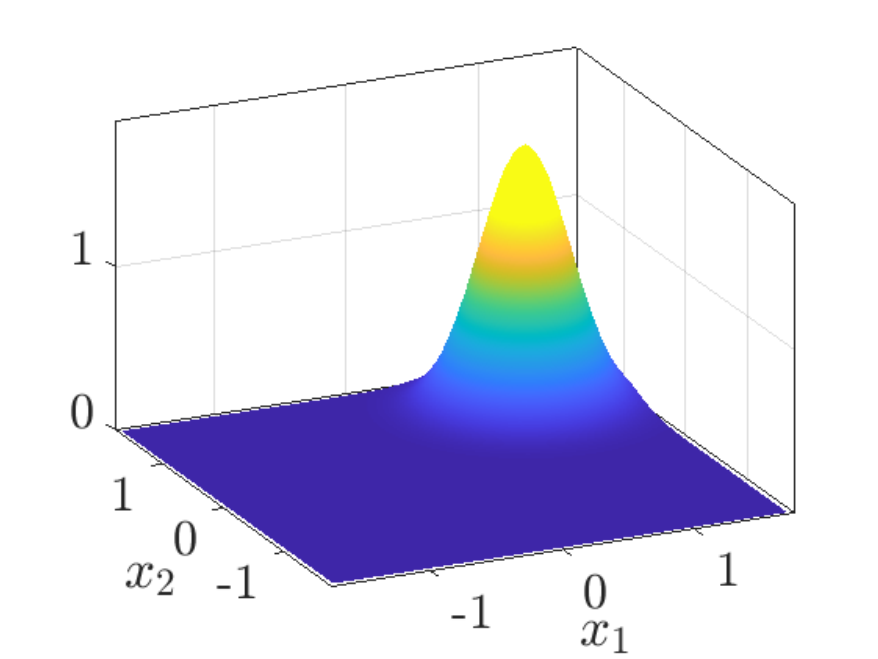}
    \end{minipage}
    \begin{minipage}[t]{0.24\textwidth}
        \centering
        \subfiguretitle{(b) $ t = 10 $}
        \includegraphics[width=0.98\textwidth]{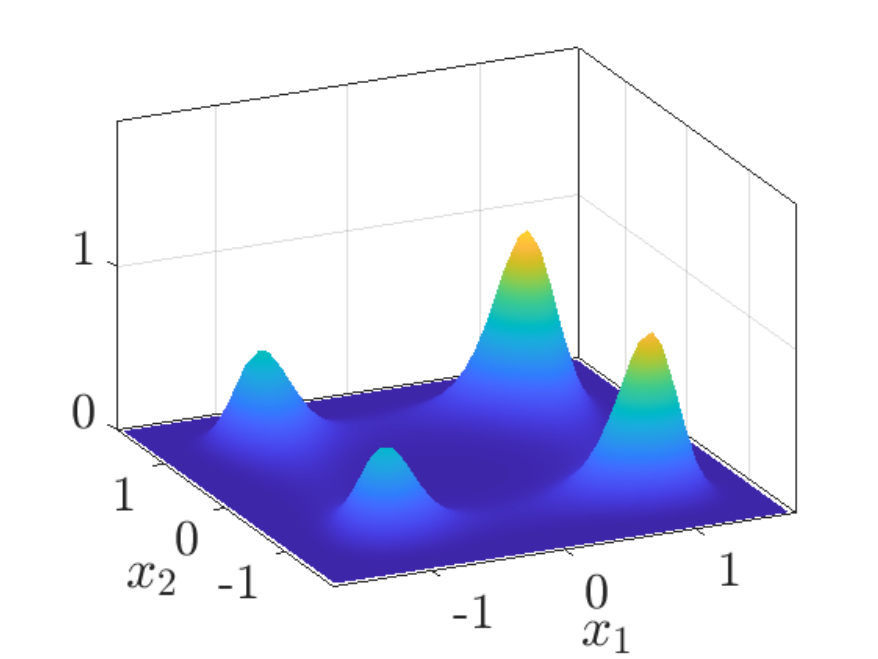}
    \end{minipage}
    \begin{minipage}[t]{0.24\textwidth}
        \centering
        \subfiguretitle{(c) $ t = 20 $}
        \includegraphics[width=0.98\textwidth]{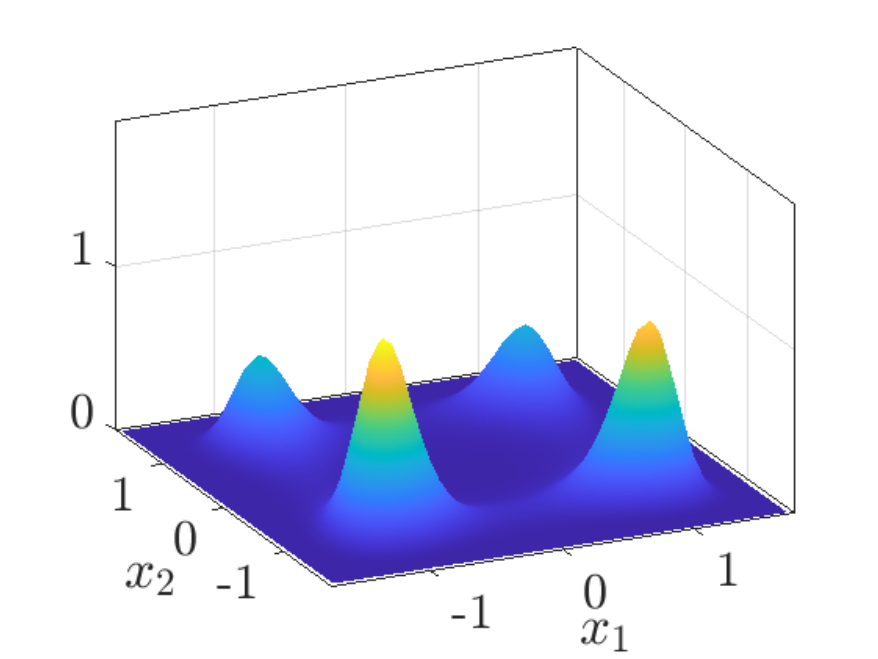}
    \end{minipage}
    \begin{minipage}[t]{0.24\textwidth}
        \centering
        \subfiguretitle{(d) $ t \to \infty $}
        \includegraphics[width=0.98\textwidth]{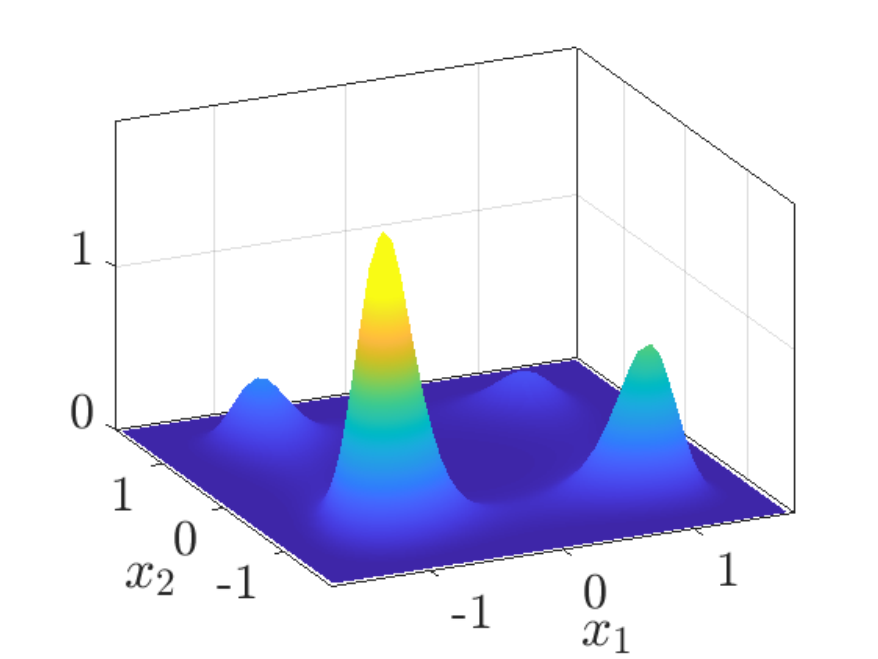}
    \end{minipage} \\[1.5ex]
    \begin{minipage}[t]{0.24\textwidth}
        \centering
        \subfiguretitle{(e) $ t = 0 $}
        \includegraphics[width=0.98\textwidth]{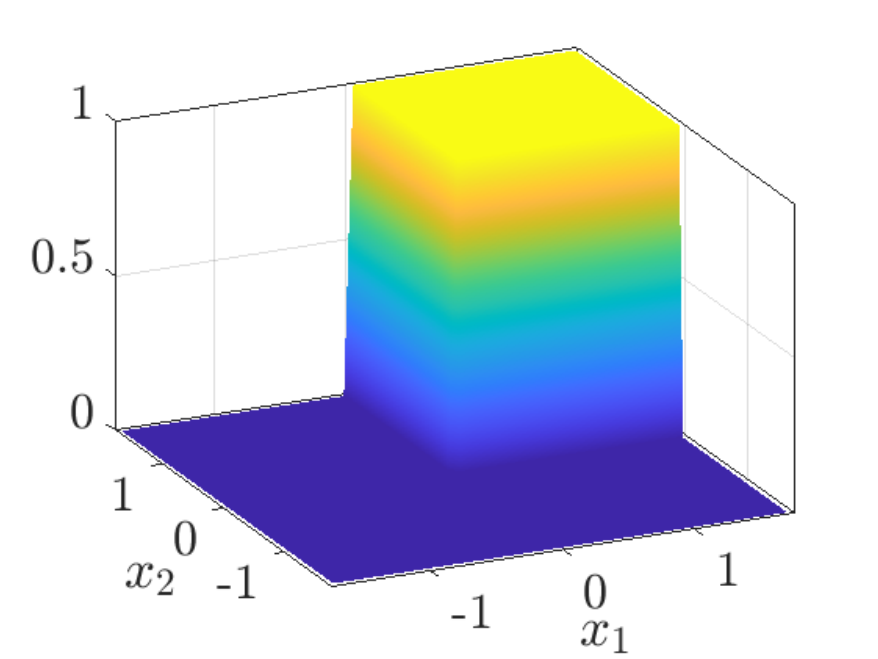}
    \end{minipage}
    \begin{minipage}[t]{0.24\textwidth}
        \centering
        \subfiguretitle{(f) $ t = 10 $}
        \includegraphics[width=0.98\textwidth]{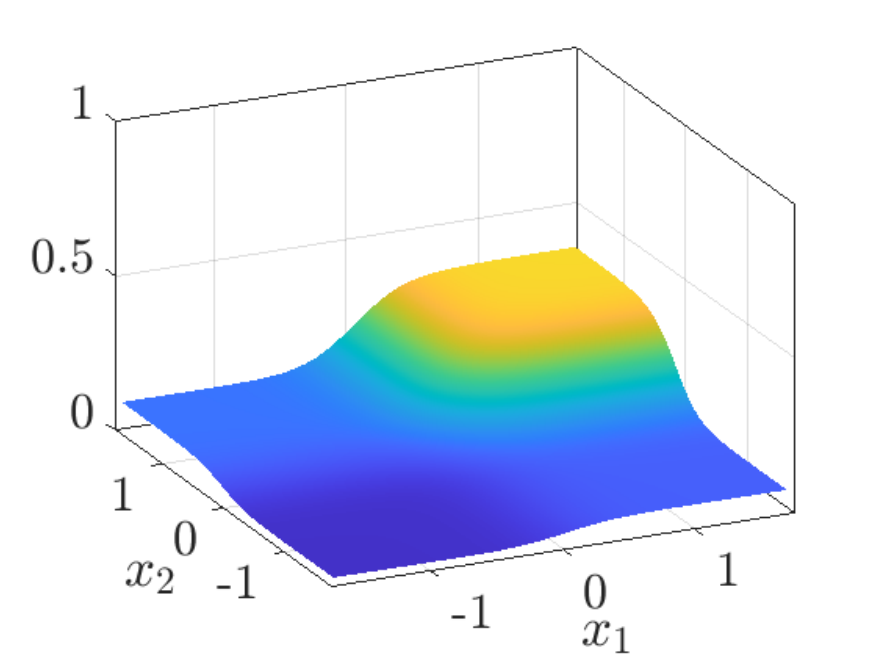}
    \end{minipage}
    \begin{minipage}[t]{0.24\textwidth}
        \centering
        \subfiguretitle{(g) $ t = 20 $}
        \includegraphics[width=0.98\textwidth]{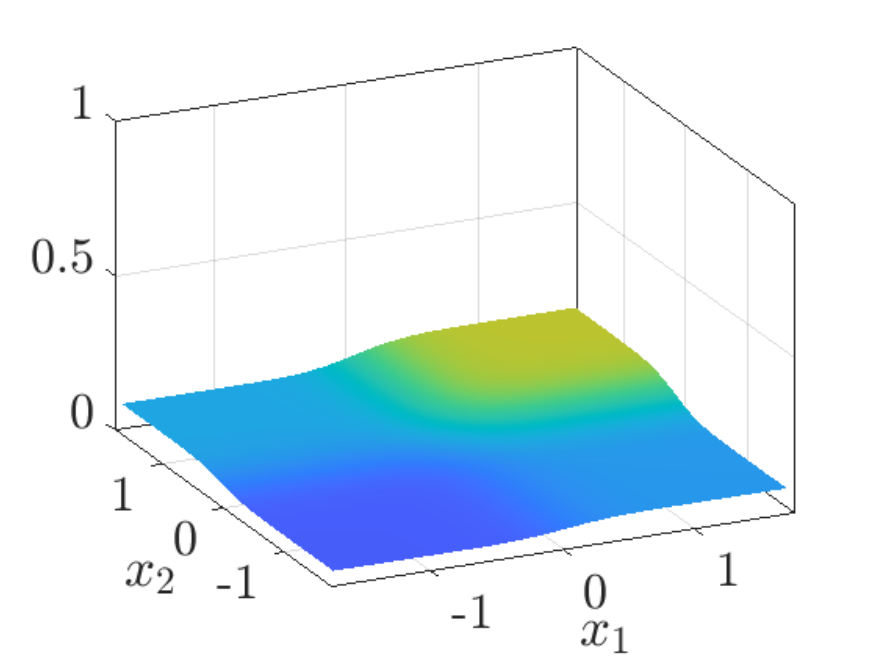}
    \end{minipage}
    \begin{minipage}[t]{0.24\textwidth}
        \centering
        \subfiguretitle{(h) $ t \to \infty $}
        \includegraphics[width=0.98\textwidth]{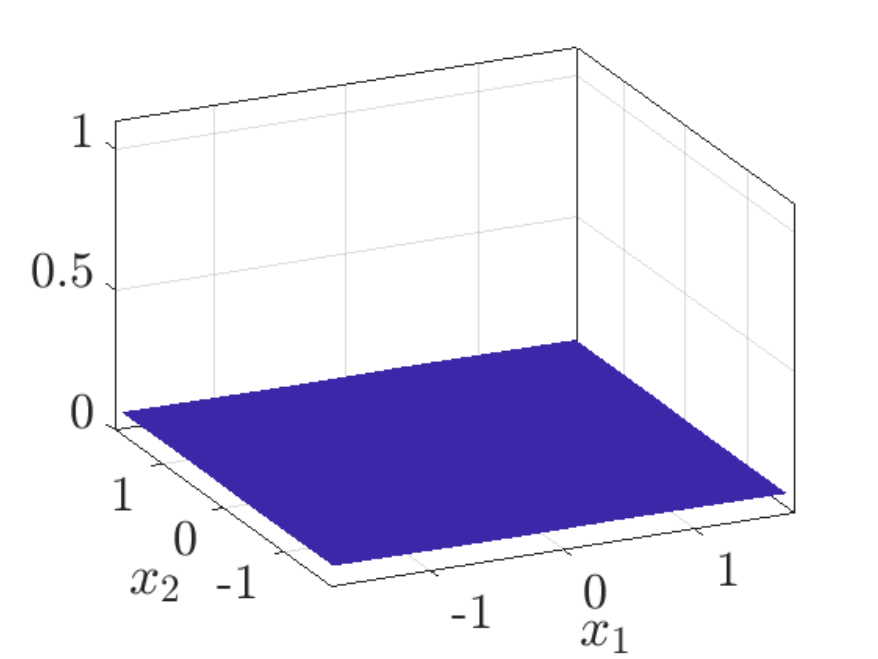}
    \end{minipage}
    \caption{(a)--(d) Illustration of the Perron--Frobenius operator. At the beginning, we assume that the system is in the upper right well with probability close to $ 1 $. With increasing time $ t $, the probability to be in the other wells increases and converges to the invariant density $ \pi $ (defined below). (e)--(h) Illustration of the Koopman operator. The observable $ f $ is the indicator function for the upper right quadrant. The Koopman operator describes the expected value of observing the system in this quadrant at time $ t $, given that it started in $ x = [x_1, x_2]^\top $. For $ t \to \infty $, $ \mathcal{K}^t f $ converges to a constant function. That is, the expected value---which is small here since most trajectories will leave the shallowest well---does not depend on the initial position anymore.}
    \label{fig:transfer operator illustration}
\end{figure}

\end{example}

Since the function spaces $ L_\mu^2 $ associated with graphs are $ \mc{n} $-dimensional, we can represent transfer operators by $ \mc{n} \times \mc{n} $ matrices. Let $ S \in \R^{\mc{n} \times \mc{n}} $ be the row-stochastic matrix \eqref{eq:transition matrix}  and define $ \boldsymbol{\rho} = [\rho(\mc[1]{x}), \dots, \rho(\mc[\mc{n}]{x})]^\top \in \R^{\mc{n}} $, then
\begin{equation*}
    \mathcal{P} \rho(\mc[i]{x}) = \big[S^\top \boldsymbol{\rho}\big]_i,
\end{equation*}
i.e., $ S^\top $ can be viewed as a matrix representation $ P $ of the Perron--Frobenius operator $ \mathcal{P} $. We obtain the matrix representations
\begin{align*}
    P = S^\top, \quad
    K = S, \quad
    T = D_\nu^{-1} \ts S^\top D_\mu, \quad
    F = S \ts D_\nu^{-1} \ts S^\top D_\mu, \quad
    B = D_\nu^{-1} \ts S^\top D_\mu \ts S
\end{align*}
of the corresponding operators, where
\begin{equation*}
    D_\mu = \diag\big(\mu(\mc[1]{x}), \dots, \mu(\mc[\mc{n}]{x})\big)
    \quad \text{and} \quad
    D_\nu = \diag\big(\nu(\mc[1]{x}), \dots, \nu(\mc[\mc{n}]{x})\big)
\end{equation*}
are invertible since we assumed the densities $ \mu $ and $ \nu $ to be strictly positive. Properties of transfer operators will be demonstrated in more detail in Section~\ref{sec:Graph representations of stochastic processes}.

\subsection{Infinitesimal generators}

In addition to analyzing spectral properties of transfer operators, we can also consider the associated infinitesimal generators. Given a time-homogeneous stochastic differential equation of the form \eqref{eq:SDE}, the Koopman operators $ \{ \mathcal{K}^{\ts\tau} \}_{\tau \ge 0} $ form a so-called one-parameter semigroup of operators.

\begin{definition}[Infinitesimal generator]
The \emph{infinitesimal generator} of the semigroup of operators is defined by
\begin{equation*}
    \mathcal{L} f = \lim_{\tau \rightarrow 0} \frac{1}{\tau} \left(\mathcal{K}^{\ts\tau} f - f \right),
\end{equation*}
see \cite{LaMa94} for more details.
\end{definition}

It can be shown, using It\^{o}'s lemma, that the infinitesimal generator associated with \eqref{eq:SDE} is given by
\begin{equation*}
    \mathcal{L} f = \sum_{i=1}^d b_i \ts \pd{f}{x_i} + \frac{1}{2} \sum_{i=1}^d \sum_{j=1}^d a_{ij} \ts \pd{^2 f}{x_i \ts \partial x_j},
\end{equation*}
where $ a = \sigma \ts \sigma^\top $. Its adjoint, the generator of the Perron--Frobenius operator, can be written as
\begin{equation*}
    \mathcal{L}^* \rho = -\sum_{i=1}^d \pd{(b_i \ts \rho)}{x_i}  + \frac{1}{2} \sum_{i=1}^d \sum_{j=1}^d \pd{^2 (a_{ij} \ts \rho)}{x_i \ts \partial x_j}.
\end{equation*}
The second-order partial differential equations $ \pd{f}{t} = \mathcal{L} f $ and $ \pd{\rho}{t} = \mathcal{L}^* \rho $ are called \emph{Kolmogorov backward equation} and \emph{Fokker--Planck equation}, respectively. Due to the spectral mapping theorem, see, e.g., \cite{Pazy83}, if $ \lambda $ is an eigenvalue of the generator, then $ e^{\lambda \ts \tau} $ is an eigenvalue of the corresponding operator with lag time $ \tau $ and the eigenfunctions are identical.

\begin{definition}[Probability flux] \label{def:probability flux}
We can also write
\begin{equation*}
    \mathcal{L}^* \rho = -\nabla \vdot J(\rho), \quad \text{with } J(\rho) = b \ts \rho - \frac{1}{2} \nabla \vdot (a \ts \rho),
\end{equation*}
where $ J $ is called the \emph{probability flux} \cite{Pav14}.
\end{definition}

\begin{example}
The infinitesimal generators associated with the Langevin equation \eqref{eq:overdamped Langevin} can be written as
\begin{equation*}
    \mathcal{L} f = -\nabla V \vdot \nabla f + \beta^{-1} \Delta f
    \quad \text{and} \quad
    \mathcal{L}^* \rho = \Delta V \ts \rho + \nabla V \vdot \nabla \rho + \beta^{-1} \Delta \rho.
\end{equation*}
The probability flux is given by $ J(\rho) = -\nabla V \ts \rho - \beta^{-1} \nabla \rho $. \exampleSymbol
\end{example}

A discrete analogue of the Kolmogorov backward equation for random-walk processes on graphs is the system of linear ordinary differential equations
\begin{equation*}
    \frac{\mathrm{d}}{\mathrm{d}t} f = L \ts f, ~~ \text{where} ~~ L(x, y) \ge 0 ~~ \forall x \neq y \quad \text{and} \quad L(x, x) = -\sum_{\mathclap{y\in \mc{X}\setminus\{x\}}} L(x, y),
\end{equation*}
which describes consensus dynamics, see, e.g., \cite{Saber03}. Here, $ L $ is often called a \emph{rate matrix}, where the off-diagonal entries $ L(x, y) $, with $ x \neq y $, represent the \emph{transition rates}, i.e., the average number of transitions from $ x $ to $ y $ per time unit. Similarly, the values $ -L(x, x) $ correspond to the \emph{escape rates} that determine the waiting times of a Markov jump process; the expected waiting time in a node $ x $ is $|{L(x, x)}|^{-1}$. The rate matrix $ L $ generates a whole family of transition matrices
\begin{equation*}
    S^\tau = \exp(\tau \ts L), \quad \tau \geq 0.
\end{equation*}
Analogously, the differential equation
\begin{equation*}
    \frac{\mathrm{d}}{\mathrm{d}t} \rho = L^* \ts \rho
\end{equation*}
can be regarded as a continuous-time version of a random walk on a graph and corresponds to the Fokker--Planck equation, see \cite{LaSch22}. In the field of spectral graph theory, specific choices of the matrix $ L $ have been shown to capture important graph characteristics, e.g., the combinatorial graph Laplacian or the normalized graph Laplacian. We establish a connection with the \textit{random-walk graph Laplacian} that is defined as
\begin{equation*}
    L_{\text{rw}}(x, y) = \begin{cases} 1, & x = y, \\ -\frac{w(x, y)}{d(x)}, & x \neq y, (x, y)\in\mc{E}, \\ 0, & \text{otherwise}, \end{cases}
\end{equation*}
so that $ L = -L_{\text{rw}} $ is a rate matrix. Here, the expected waiting time in a node is proportional to its degree, which implies that the process is slower in nodes with high degrees, such as densely connected clusters. Other variants of time-continuous random walks have been considered in \cite{Djurdjevac12}.

Since it holds that $ L_{\text{rw}} = I - K $, computing the largest eigenvalues of the matrix $ K $ is equivalent to computing the smallest eigenvalues of the random-walk graph Laplacian. This shows that conventional spectral clustering relies on the dominant eigenfunctions of the Koopman operator and can be interpreted in terms of metastable sets. A natural generalization of spectral clustering to directed and time-evolving graphs is thus to use the forward--backward operator so that the resulting clusters can be viewed as coherent sets. We call the matrix $ L_{\text{fb}} = I - F $ \emph{forward--backward Laplacian}. Detailed derivations and applications can be found in \cite{KD22, KT24}.

\subsection{Invariant density and reversibility}\label{subsec:InvariantDistr}

In what follows, we will consider the reweighted  Perron--Frobenius operator $ \mathcal{T} $ with respect to different probability densities. A particularly important role is played by the invariant density.

\begin{definition}[Invariant density]
A density $ \pi $ is called \emph{invariant} if $ \mathcal{P} \pi = \pi $. That is, $ \pi $ is an eigenfunction of the Perron--Frobenius operator corresponding to the eigenvalue $ \lambda = 1 $.
\end{definition}

\begin{example}
\setlength{\belowdisplayskip}{4pt} Let us consider two simple examples:
\begin{enumerate}[leftmargin=3.5ex, itemsep=0ex, topsep=0.5ex, label=\roman*)]
\item The invariant density of the overdamped Langevin equation \eqref{eq:overdamped Langevin} is given by
\begin{equation*}
    \pi(x) = \frac{1}{Z} e^{-\beta \ts V(x)},
    \quad \text{with } Z = \intop_\mathbb{X} e^{-\beta \ts V(x)} \ts \mathrm{d}x,
\end{equation*}
since it can be easily shown that $ J(\pi) = 0 $ and thus $ \mathcal{L}^* \pi = -\nabla \vdot J(\pi) = 0 $. That is, $ \pi $ is an eigenfunction of $ \mathcal{L}^* $ associated with the eigenvalue $ \lambda = 0 $ and hence an eigenfunction of $ \mathcal{P}^{\ts\tau} $ associated with the eigenvalue $ \lambda = 1 $.
\item Given a connected undirected graph $ \mc{G} $, the invariant density is given by
\begin{equation*}
    \pi(x) = \frac{1}{Z} d(x),
    \quad \text{with } Z = \sum_{x \in \mc{X}} d(x),
\end{equation*}
since
\begin{equation*}
    \mathcal{P} \pi(x) = \sum_{y \in \mc{X}} p(y, x) \ts \pi(y) = \sum_{y \in \mc{X}} \frac{w(y, x)}{d(y)} \frac{1}{Z} d(y) = \frac{1}{Z} \sum_{y \in \mc{X}} w(x, y) = \frac{1}{Z} d(x) = \pi(x),
\end{equation*}
where we used the fact that $ w(x, y) = w(y, x) $ since the graph is undirected. \exampleSymbol
\end{enumerate}
\end{example}

If the reference density $ \mu $ is the invariant density, i.e., $ \mu = \pi $, this immediately implies that $ \nu = \pi $ and the operator $ \mathcal{T} $ can be written as
\begin{tcolorbox}[sidebyside align=bottom]
$
    \displaystyle \mathcal{T}^{\ts\tau} u(x) = \frac{1}{\pi(x)} \intop_\mathbb{X} p_\tau(y, x) \ts \pi(y) \ts u(y) \ts \mathrm{d} y,
$
\tcblower
$
    \displaystyle \mathcal{T} u(x) = \frac{1}{\pi(x)} \sum_{y \in \mc{X}} p(y, x) \ts \pi(y) \ts u(y).
$
\end{tcolorbox}
The Perron--Frobenius operator reweighted with respect to the invariant density $ \pi $ will be considered in more detail below.

\begin{definition}[Reversibility]
The process is called \emph{reversible} if there exists a probability density $ \pi $ that satisfies the \emph{detailed balance condition}
\begin{tcolorbox}[sidebyside align=bottom, after skip=10pt]
$
    \pi(x) \ts p_{\tau}(x, y) = \pi(y) \ts p_{\tau}(y,x) ~~ \forall x, y \in \mathbb{X},
$
\tcblower
$
    \pi(x) \ts p(x, y) = \pi(y) \ts p(y, x) ~~ \forall x, y \in \mc{X}.
$
\end{tcolorbox}
\end{definition}

\begin{proposition}
A stochastic process with generator $ \mathcal{L} $ is reversible if and only if the stationary probability flux $ J(\pi) $ vanishes \cite{Pav14}.
\end{proposition}

\begin{example} \label{ex:reversible systems}
Let us again consider the guiding examples:
\begin{enumerate}[leftmargin=3.5ex, itemsep=0ex, topsep=0.5ex, label=\roman*)]
\item We have seen above that the probability flux $ J(\pi) $ for the overdamped Langevin equation is identically zero so that the process is reversible. Note that this is a stronger property than $ \mathcal{L}^* \pi = -\nabla \vdot J(\pi) = 0 $, which only implies that the probability flux is divergence-free.
\item A random walk process on an undirected graph is reversible since
\begin{equation*}
    \pi(x) \ts p(x, y) = \frac{1}{Z} d(x) \frac{w(x, y)}{d(x)} = \frac{1}{Z} d(y) \frac{w(y, x)}{d(y)} = \pi(y) \ts p(y, x). \tag*{\exampleSymbol}
\end{equation*}
\end{enumerate}
\end{example}

Using the detailed balance condition, we have
\begin{tcolorbox}
\setlength{\abovedisplayskip}{0pt}
\setlength{\belowdisplayskip}{0pt}
\begin{align*}
    \mathcal{T}_\tau f(x)
        &= \frac{1}{\pi(x)} \intop_\mathbb{X} p_{\tau}(y, x) \ts \pi(y) \ts f(y)\,\mathrm{d}y \\
        &= \intop_\mathbb{X} p_\tau(x, y) \ts f(y) \ts \mathrm{d} y \\
        &= \mathcal{K}_\tau f(x),
\end{align*}
\tcblower
\setlength{\abovedisplayskip}{0pt}
\setlength{\belowdisplayskip}{0pt}
\begin{align*}
    \mathcal{T} f(x)
            &= \frac{1}{\pi(x)} \sum_{y \in \mc{X}} p(y, x) \ts \pi(y) \ts f(y) \\
            &= \sum_{y \in \mc{X}} p(x, y) \ts f(y) \\
            &= \mathcal{K} f(x).
\end{align*}
\end{tcolorbox}
This shows that for reversible dynamical systems the transfer operator $ \mathcal{T} $ reweighted with respect to the invariant density $ \pi $ is identical to the Koopman operator $ \mathcal{K} $. Furthermore, we obtain
\begin{align*}
    \innerprod{\mathcal{P}f}{g}_{\pi^{-1}} = \innerprod{f}{\mathcal{P} g}_{\pi^{-1}}
    \quad \text{and} \quad
    \innerprod{\mathcal{K}f}{g}_\pi = \innerprod{f}{\mathcal{K}g}_\pi.
\end{align*}
That is, if the stochastic process is reversible, $ \mathcal{P} $ and $ \mathcal{K} $ are self-adjoint, which, assuming the operators are compact, implies that the eigenvalues are real-valued and the eigenfunctions form an orthogonal basis with respect to the respective reweighted inner products \cite{SKH23}. Metastability is then reflected in the spectral properties of transfer operators. Information about metastable sets is encoded in the eigenfunctions $ \varphi_\ell $ corresponding to the dominant eigenvalues $ \lambda_\ell \approx 1 $ since then $ \mathcal{K} \varphi_\ell = \lambda_\ell \ts \varphi_\ell \approx \varphi_\ell $. Dynamics associated with small eigenvalues, on the other hand, are damped out quickly. For more details on metastability, see Appendix \ref{app:Metastability} and the references therein.

\subsection{Non-reversible processes}

One approach to accelerate the convergence of a process to the invariant density is to depart from reversible dynamics \cite{LNP13}. The degree or irreversibility of the system can be quantified through the entropy production rate, which has been studied both for non-equilibrium Markov processes \cite{jiang2003} as well as for directed graphs \cite{Lambiotte24} and their cycle decompositions \cite{DjConradBanischSchuette15}.

\begin{example}
Let us modify the examples considered above:
\begin{enumerate}[wide, labelwidth=!, labelindent=0pt, itemsep=0ex, topsep=0.5ex, label=\roman*)]
\item A simple way to construct a non-reversible system is to add a divergence-free term (with respect to the invariant distribution), e.g.,
\begin{equation} \label{eq:non-reversible Langevin}
    \mathrm{d}X_t = (-\nabla V(X_t) + M \ts \nabla V(X_t)) \ts \mathrm{d}t + \sqrt{2 \beta^{-1}} \ts \mathrm{d}W_t,
\end{equation}
where $ M $ is an antisymmetric matrix. The invariant density of this SDE is again $ \pi(x) = \frac{1}{Z} e^{-\beta \ts V(x)} $ since the probability flux $ J(\pi) = \frac{1}{Z} M \ts \nabla V \ts e^{-\beta \ts V} $ is divergence-free, but it does not vanish. That is, the process admits an invariant distribution, but is not reversible.
\item Given an undirected graph $ \mc{G} = (\mc{X}, \mc{E}, w) $ with invariant density $ \pi $, it is possible to construct a directed graph $ \widetilde{\mc{G}} = (\mc{X}, \widetilde{\mc{E}}, \widetilde{w}) $ with the same invariant density by defining
\begin{equation*}
    \widetilde{w}(x, y) = w(x, y) + m(x, y), \quad \text{with }
    \sum_{x \in \mc{X}} m(x, y) = \sum_{y \in \mc{X}} m(x, y) = 0,
\end{equation*}
since then $ \widetilde{d}(x) = d(x) $ and $ \widetilde{p}(x, y) = p(x, y) + \frac{m(x, y)}{d(x)} $ so that
\begin{equation*}
    \widetilde{\mathcal{P}} \pi(x) = \sum_{y \in \mc{X}} \widetilde{p}(y, x) \ts \pi(y) = \sum_{y \in \mc{X}} p(y, x) \ts \pi(y) + \sum_{y \in \mc{X}} \frac{m(y, x)}{d(y)} \ts \pi(y) = \pi(x),
\end{equation*}
where $ \widetilde{\mathcal{P}} $ is the Perron--Frobenius operator associated with the modified graph $ \widetilde{\mc{G}} $. The second sum is zero as $ \frac{\pi(y)}{d(y)} = \frac{1}{Z} $ and the $ m(y, x) $ terms sum up to zero. \exampleSymbol
\end{enumerate}
\end{example}

If the stochastic differential equation is non-reversible or, analogously, the graph is not undirected, then the eigenvalues of the Perron--Frobenius and Koopman operators are in general complex-valued and conventional methods to detect metastable sets or clusters typically fail, see \cite{DjCWS16, KD22} for more details. To circumvent this problem, a common approach is to consider eigenvalues and eigenfunctions of the forward--backward and backward--forward operators instead.

\begin{theorem} \label{thm:spectral properties}
The eigenvalues of $ \mathcal{F} $ and $ \mathcal{B} $ are contained in $ [0, 1] $.
\end{theorem}

\begin{proof}
The proof requires auxiliary results that are contained in Appendix~\ref{app:Properties of transfer operators}. Corollary~\ref{cor:self-adjoint and positive definite} implies that the eigenvalues of $ \mathcal{F} $ and $ \mathcal{B} $ are real-valued and non-negative. Combining this with Lemma~\ref{lem:spectral radius} concludes the proof.
\end{proof}

A set $ A $ is called \emph{coherent} from time $ t $ to $ t + \tau $ if $ \Theta_{t,\tau} (\Theta_{t+\tau,-\tau} (A)) \cong A $. That is, the set $ A $ is almost invariant under the forward--backward dynamics. In order to detect such coherent sets, we compute eigenfunctions $ \varphi_\ell $ corresponding to large eigenvalues $ \lambda_\ell \approx 1 $ of the forward--backward operator $ \mathcal{F} $ so that $ \mathcal{F} \varphi_\ell = \lambda_\ell \ts \varphi_\ell \approx \varphi_\ell $. Let $ \psi_\ell := \frac{1}{\sqrt{\lambda_\ell}} \mathcal{T} \varphi_\ell $, then
\begin{equation*}
    \mathcal{K} \psi_\ell = \sqrt{\lambda_\ell} \ts \varphi_\ell, \quad
    \mathcal{T} \varphi_\ell = \sqrt{\lambda_\ell} \ts \psi_\ell, \quad
    \mathcal{F} \varphi_\ell = \lambda_\ell \ts \varphi_\ell, \quad
    \mathcal{B} \psi_\ell = \lambda_\ell \ts \psi_\ell.
\end{equation*}
This shows that the eigenfunctions of the forward--backward operator can also be interpreted as the singular functions of the reweighted Perron--Frobenius operator and the eigenfunctions of the backward--forward operator as the singular functions of the Koopman operator. More details about singular value decompositions of compact transfer operators can be found in \cite{MSKS20}.

\section{Approximation of transfer operators}
\label{sec:Approximation of transfer operators}

Transfer operators associated with stochastic processes defined on continuous state spaces are infinite-dimensional. In practice, however, we have to restrict ourselves to finite-dimensional subspaces. Furthermore, we also typically have to estimate approximations of these operators from data. Graph transfer operators, on the other hand, are already finite-dimensional and can be easily computed from the adjacency matrices.

\subsection{Galerkin approximation of transfer operators}

Let $ \{ \phi_i \}_{i=1}^n $ be a set of $ n $ linearly independent basis functions---e.g., monomials, radial basis functions, trigonometric functions, or indicator functions---spanning an $ n $-dimensional subspace of the domain of an arbitrary operator $ \mathcal{A} $ that we want to approximate and
\begin{equation*}
    \phi(x) = [\phi_1(x), \dots, \phi_n(x)]^\top \in \R^n.
\end{equation*}
That is, any function $ f $ in this subspace can be written as
\begin{equation*}
    f(x) = \sum_{i=1}^n c_i \ts \phi_i(x) = c^\top \phi(x),
\end{equation*}
where $ c = [c_1, \dots, c_n]^\top \in \R^n $. The \emph{Galerkin projection} $ \mathcal{A}_\phi $ of the operator $ \mathcal{A} $ onto the $ n $-dimensional subspace is defined by the matrix
\begin{equation*}
    A_\phi = \big[G_0(\mu)\big]^{-1} G_1({\mathcal{A}, \mu}) \in \R^{n \times n},
\end{equation*}
with entries
\begin{equation*}
    \big[G_0(\mu)\big]_{ij} = \innerprod{\phi_i}{\phi_j}_\mu
    \quad \text{and} \quad
    \big[G_1(\mathcal{A}, \mu)\big]_{ij} = \innerprod{\phi_i}{\mathcal{A} \phi_j}_\mu.
\end{equation*}
We then define the projected operator
\begin{equation*}
    \mathcal{A}_\phi f(x) := (A_\phi \ts c)^\top \phi(x).
\end{equation*}
Eigenfunctions of $ \mathcal{A} $ can hence be approximated by eigenfunctions of $ \mathcal{A}_\phi $. Let $ \xi_\ell $ be an eigenvector of $ A_\phi $ with associated eigenvalue $ \lambda_\ell $, then defining $ \varphi_\ell(x) = \xi_\ell^\top \phi(x) $ yields
\begin{equation*}
    \mathcal{A}_\phi \varphi_\ell(x) = (A_\phi \ts \xi_\ell)^\top \phi(x) = \lambda_\ell \ts \xi_\ell^\top \phi(x) = \lambda_\ell \ts \varphi_\ell(x).
\end{equation*}
Lemma~\ref{lem:adjointness} immediately implies that $ G_1(\mathcal{T}, \nu) = G_1(\mathcal{K}, \mu)^\top $ since
\begin{equation*}
    \big[G_1(\mathcal{T}, \nu)\big]_{ij} = \innerprod{\phi_i}{\mathcal{T} \phi_j}_\nu = \innerprod{\mathcal{K} \phi_i}{\phi_j}_\mu = \big[G_1(\mathcal{K}, \mu)\big]_{ji}.
\end{equation*}

\begin{remark}
Although the state space of graph transfer operators is already finite-dimensional and we can easily construct the matrix representations of the operators, a Galerkin approximation can still be used to project the dynamics given by the random walk process onto a lower-dimensional space. This reduces the computational complexity of the resulting eigenvalue problems. A open problem, however, is how to choose the basis functions in such a way that the dominant spectrum and thus the cluster structure is retained, see \cite{KT24} for further details and examples.
\end{remark}

\subsection{Data-driven approximation of transfer operators}

Typically, the integrals required for the Galerkin approximation are estimated from trajectory data using Monte Carlo integration. Given training data $ \big\{ (x^{(k)}, y^{(k)}) \big\}_{k=1}^m $, where $ x^{(k)} $ is sampled from the distribution $ \mu $ and $ y^{(k)} = \Theta_\tau\big(x^{(k)}\big) $ is hence sampled from the distribution~$ \nu $, we define
\begin{equation*}
    \Phi_x = \begin{bmatrix} \phi(x^{(1)}), & \phi(x^{(2)}), & \dots, & \phi(x^{(m)}) \end{bmatrix}, \quad
    \Phi_y = \begin{bmatrix} \phi(y^{(1)}), & \phi(y^{(2)}), & \dots, & \phi(y^{(m)}) \end{bmatrix},
\end{equation*}
and the matrices
\begin{alignat*}{2}
    C_{xx} &= \frac{1}{m} \Phi_x \ts \Phi_x^\top, &\qquad
    C_{yy} &= \frac{1}{m} \Phi_y \ts \Phi_y^\top, \\
    C_{xy} &= \frac{1}{m} \Phi_x \ts \Phi_y^\top, &
    C_{yx} &= \frac{1}{m} \Phi_y \ts \Phi_x^\top
\end{alignat*}
so that
\begin{alignat*}{2}
    \big[C_{xx}\big]_{ij} &= \frac{1}{m} \sum_{k=1}^m \phi_i(x_k) \ts \phi_j(x_k) &&\underset{\scriptscriptstyle m \rightarrow \infty}{\longrightarrow} \innerprod{\phi_i}{\phi_j}_\mu = \big[G_0(\mu)\big]_{ij}, \\
    \big[C_{yy}\big]_{ij} &= \frac{1}{m} \sum_{k=1}^m \phi_i(y_k) \ts \phi_j(y_k) &&\underset{\scriptscriptstyle m \rightarrow \infty}{\longrightarrow} \innerprod{\phi_i}{\phi_j}_\nu = \big[G_0(\nu)\big]_{ij} \\
    \big[C_{xy}\big]_{ij} &= \frac{1}{m} \sum_{k=1}^m \phi_i(x_k) \ts \phi_j(y_k) &&\underset{\scriptscriptstyle m \rightarrow \infty}{\longrightarrow}
    \innerprod{\phi_i}{\mathcal{K}^{\ts\tau} \phi_j}_\mu = \big[G_1(\mathcal{\mathcal{K}^{\ts\tau}}, \mu)\big]_{ij}, \\
    \big[C_{yx}\big]_{ij} &= \frac{1}{m} \sum_{k=1}^m \phi_i(y_k) \ts \phi_j(x_k) &&\underset{\scriptscriptstyle m \rightarrow \infty}{\longrightarrow} \innerprod{\mathcal{K}^{\ts\tau} \phi_i}{\phi_j}_\mu = \big[G_1(\mathcal{\mathcal{T}^{\ts\tau}}, \nu)\big]_{ij},
\end{alignat*}
see also \cite{SKH23}. That is, we can approximate the Galerkin projections of $ \mathcal{K} $ and $ \mathcal{T} $ and hence also of $ \mathcal{F} $ and $ \mathcal{B} $ from data. The matrix representations of the estimated Galerkin projections are
\begin{equation*}
    K_\phi \approx C_{xx}^{-1} C_{xy}, \quad
    T_\phi \approx C_{yy}^{-1} C_{yx}, \quad
    F_\phi \approx C_{xx}^{-1} C_{xy} \ts C_{yy}^{-1} C_{yx}, \quad
    B_\phi \approx C_{yy}^{-1} C_{yx} \ts C_{xx}^{-1} C_{xy}.
\end{equation*}
This data-driven estimation of the Koopman operator, later generalized to other transfer operators, is called \emph{extended dynamic mode decomposition} (EDMD) \cite{WKR15, KKS16}.

\begin{remark}
Over the last years, a host of different methods for approximating transfer operators from trajectory data have been proposed. EDMD is a nonlinear variant of \emph{dynamic mode decomposition} \cite{Schmid10}, which is closely related to \emph{time-lagged independent component analysis} \cite{MS94} and \emph{linear inverse modeling} \cite{Penland96}. In order to mitigate the curse of dimensionality---high-dimensional systems typically require a large set of basis functions---, kernel-based extensions \cite{WRK15, KSM19}, tensor-based variants \cite{Nueske2016, NGKC21}, and deep learning approaches \cite{YKH19, MPWN18} have been proposed. The optimal choice of basis functions, kernel, or neural network architecture is strongly problem-dependent and often requires tuning the hyperparameters. Poorly chosen ansatz spaces can lead to ill-conditioned matrices and spectral pollution. A detailed overview and comparison can be found in \cite{SKH23, Colbrook23}.
\end{remark}

\subsection{Ulam's method}

One of the simplest and also most popular approaches for approximating transfer operators is \emph{Ulam's method} \cite{Ulam60}, which can be regarded as a Galerkin projection of the operator onto a finite-dimensional subspace spanned by indicator functions \cite{KKS16}. Let $ \{ \mathbb{B}_1, \dots, \mathbb{B}_n \} $ be a decomposition of $ \mathbb{X} \subset \R^d $ into $ n $ disjoint sets (e.g., a box discretization of the domain), i.e.,
\begin{equation}\label{eq:fullPart}
    \bigcup_{i=1}^n \mathbb{B}_i = \mathbb{X}
    \quad \text{and} \quad
    \mathbb{B}_i \cap \mathbb{B}_j = \varnothing ~ \forall i \ne j,
\end{equation}
and $ \phi_i := \mathds{1}_{\mathbb{B}_i} $ the corresponding indicator functions defined by
\begin{equation*}
    \mathds{1}_{\mathbb{B}_i}(x) =
    \begin{cases}
        1, & x \in \mathbb{B}_i, \\
        0, & \text{otherwise}.
    \end{cases}
\end{equation*}
The resulting subspace spanned by these basis functions consists of piecewise constant functions. As the sets $ \mathbb{B}_i $ are disjoint by definition, $ G_0(\mu) $ is a diagonal matrix, whose entries are identical to the row sums of the matrix $ G_1(\mathcal{K}^{\ts\tau}, \mu) $. That is, $ K_n^{\ts\tau} = \big[G_0(\mu)\big]^{-1} G_1(\mathcal{K}^{\ts\tau}, \mu) $ is a row-stochastic matrix containing the transition probabilities between the disjoint sets. If we view the matrix $ G_1(\mathcal{K}^{\ts\tau}, \mu) $ as a weighted adjacency matrix of an \emph{induced graph}, then the matrix $ K_n^{\ts\tau} $ is the transition probability matrix~$ S $. Instead of analyzing trajectories of the time-continuous stochastic process, we then consider the corresponding coarse-grained map defined by the discrete random walk process on the induced graph as illustrated in Example~\ref{ex:undirected graph}. Note that the graph structure depends on the lag time~$ \tau $.

\section{Graph representations of stochastic processes}
\label{sec:Graph representations of stochastic processes}

In this section, we will analyze properties of transfer operators associated with different types of dynamics and their induced graph representations.

\subsection{Time-homogeneous reversible processes and undirected graphs}

We first restrict ourselves to time-homogeneous reversible stochastic processes such as the overdamped Langevin dynamics and random walks on undirected graphs (see Example~\ref{ex:reversible systems}). Approximating the Koopman operator $ \mathcal{K}^{\ts\tau} $ with respect to the $ \pi $-weighted inner product using Ulam's method, the matrix $ G_1(\mathcal{K}^\tau, \pi) $ is symmetric since
\begin{equation*}
    \big[G_1(\mathcal{K}^\tau, \pi)\big]_{ij} = \innerprod{\phi_i}{\mathcal{K}^\tau \phi_j}_\pi = \innerprod{\mathcal{K}^\tau \phi_i}{\phi_j}_\pi = \big[G_1(\mathcal{K}^\tau, \pi)\big]_{ji}
\end{equation*}
and the associated induced graph is automatically undirected.

\begin{example} \label{ex:reversible quadruple-well problem}

\begin{figure}
    \centering
    \begin{minipage}[t]{0.24\textwidth}
        \centering
        \subfiguretitle{(a)}
        \vspace*{1ex}
        \includegraphics[width=0.95\textwidth]{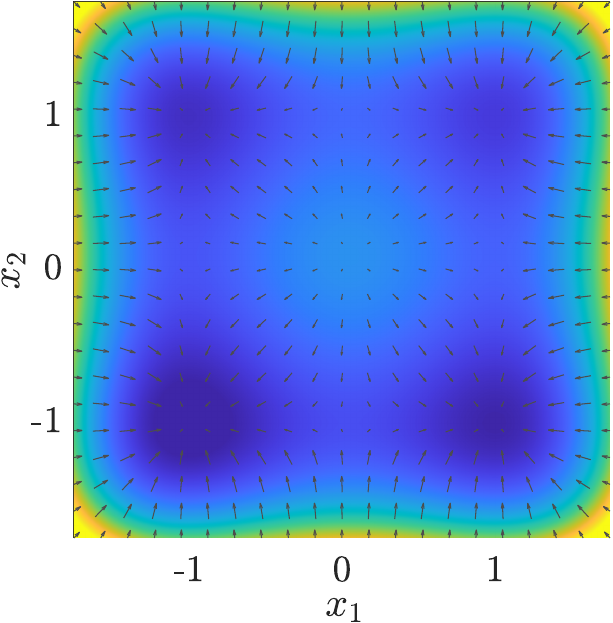}
    \end{minipage}
    \begin{minipage}[t]{0.24\textwidth}
        \centering
        \subfiguretitle{(b)}
        \vspace*{1ex}
        \includegraphics[width=0.95\textwidth]{ReversibleSDE_graph}
    \end{minipage}
    \begin{minipage}[t]{0.24\textwidth}
        \centering
        \subfiguretitle{(c)}
        \vspace*{1ex}
        \includegraphics[width=0.95\textwidth]{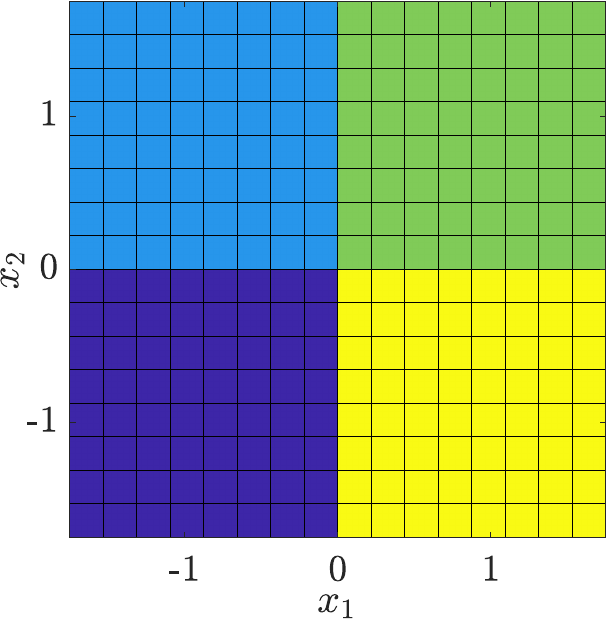}
    \end{minipage}
    \begin{minipage}[t]{0.24\textwidth}
        \centering
        \subfiguretitle{(d)}
        \vspace*{1ex}
        \includegraphics[width=\textwidth]{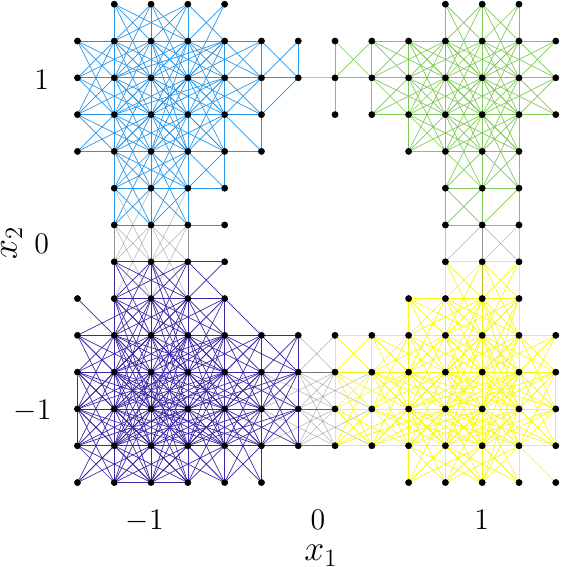}
    \end{minipage} \\[1ex]
    \begin{minipage}[t]{0.24\textwidth}
        \centering
        \subfiguretitle{(e) $ \lambda_1 = 1 $}
        \vspace*{1ex}
        \includegraphics[width=0.95\textwidth]{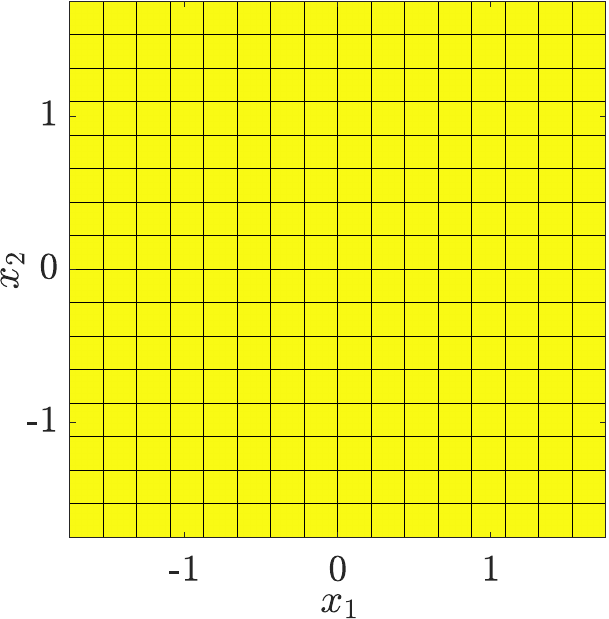}
    \end{minipage}
    \begin{minipage}[t]{0.24\textwidth}
        \centering
        \subfiguretitle{(f) $ \lambda_2 = 0.992 $}
        \vspace*{1ex}
        \includegraphics[width=0.95\textwidth]{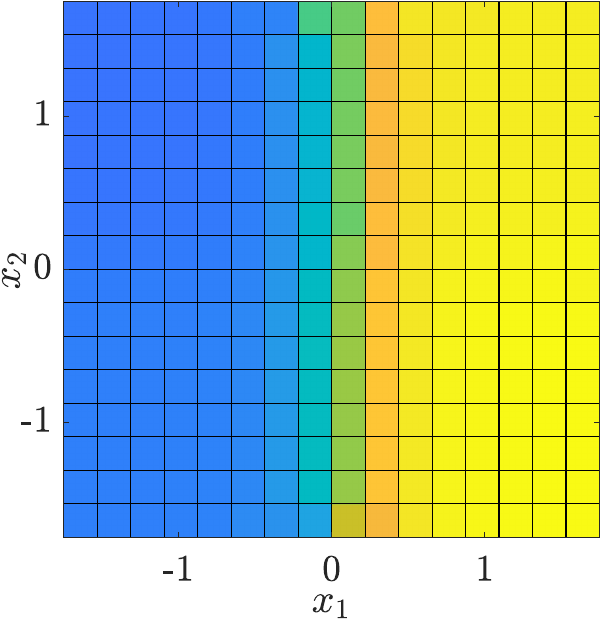}
    \end{minipage}
    \begin{minipage}[t]{0.24\textwidth}
        \centering
        \subfiguretitle{(g) $ \lambda_3 = 0.991 $}
        \vspace*{1ex}
        \includegraphics[width=0.95\textwidth]{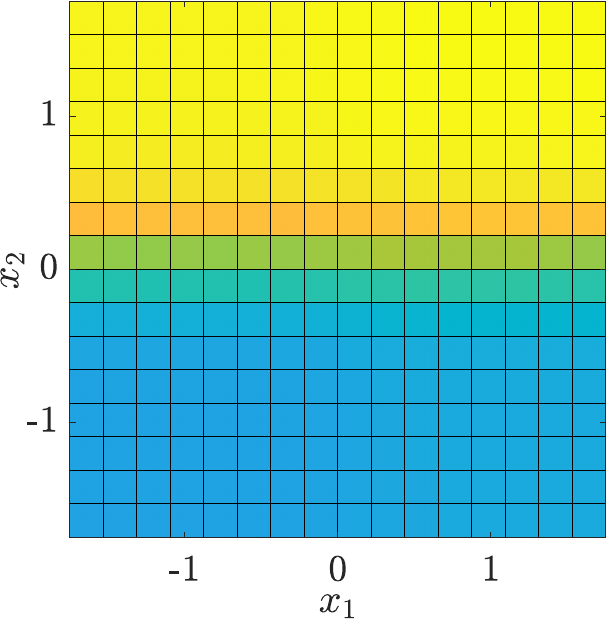}
    \end{minipage}
    \begin{minipage}[t]{0.24\textwidth}
        \centering
        \subfiguretitle{(h) $ \lambda_4 = 0.982 $}
        \vspace*{1ex}
        \includegraphics[width=0.95\textwidth]{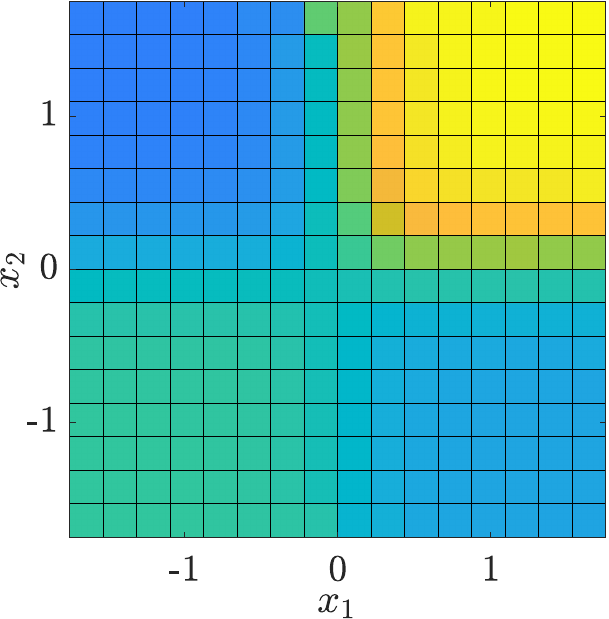}
    \end{minipage}
    \caption{(a)~Quadruple-well potential and the corresponding drift term given by the negative gradient, represented as a quiver plot. (b)~Resulting undirected graph, where we omit edges with a weight smaller than a given threshold and unconnected vertices to highlight the cluster structure. (c)~Decomposition of the domain into four metastable sets. (d)~Clustering of the graph into four weakly coupled clusters. Intra-cluster edges are colored according to the respective cluster and inter-cluster edges are shown in gray. (e)--(h) Dominant eigenfunctions of the Koopman operator associated with the stochastic differential equation.}
    \label{fig:reversible quadruple-well problem}
\end{figure}

Let us again analyze the quadruple-well problem introduced in Example~\ref{ex:quadruple-well system}. We set $ \tau = \frac{1}{10} $, subdivide the domain $ \mathbb{X} = [-1.75, 1.75] \times [-1.75, 1.75] $ into $ 16 \times 16 $ boxes of the same size, and estimate the Koopman operator from one long trajectory containing $100\ts000$ training data points using Ulam's method. A visualization of the drift term, the resulting undirected graph, and the clustering of the domain into metastable sets as well as the clustering of the graph are shown in Figure~\ref{fig:reversible quadruple-well problem}. Each well of the potential results in a cluster of highly connected vertices that is only weakly coupled to the other clusters. This is due to the metastability of the system as transitions between wells are rare events. The deepest well in the lower-left quadrant is, as expected, the most interconnected. The Koopman operator associated with the stochastic differential equation has four dominant eigenvalues, followed by a spectral gap. Applying $ k $-means to the dominant eigenfunctions results in four metastable sets. Equivalently, applying spectral clustering to the graph yields four clusters. \exampleSymbol
\end{example}

Once we have estimated the transition probabilities between the boxes using Ulam's method, we can view the computation of metastable sets as a graph clustering problem. This shows that detecting metastable sets associated with reversible processes is equivalent to finding clusters in undirected graphs.

\subsection{Time-homogeneous non-reversible processes and directed graphs}

The graphs induced by non-reversible processes are in general not undirected since the Koopman operator is not self-adjoint anymore. Let us illustrate the differences between the reversible case and the irreversible case with the aid of a simple example.

\begin{example} \label{ex:nonreversible quadruple-well problem}
Consider the quadruple-well potential introduced in Example~\ref{ex:quadruple-well system}, but now using the non-reversible stochastic differential equation \eqref{eq:non-reversible Langevin} with the antisymmetric matrix
\begin{equation*}
    M =
    \begin{bmatrix}
        0 & c \\
        -c & 0
    \end{bmatrix},
\end{equation*}
where $ c \in \R $ is a parameter. If $ c $ is large, then the additional non-reversible term, which causes curls in the vector field, dominates the dynamics. We use the same box discretization, lag time, and number of test points for estimating the transition probabilities as in the previous example. Since the process is now non-reversible, the eigenvalues of the Koopman operator $ \mathcal{K}^{\ts\tau} $ are complex-valued. Applying $ k $-means to the associated eigenfunctions (using, e.g., only the real parts), does in this case not result in the expected clustering into four metastable sets. Using the eigenvalues and eigenfunctions of the forward--backward operator~$ \mathcal{F} $ instead allows us to detect coherent sets. The numerical results are shown in Figure~\ref{fig:nonreversible quadruple-well problem}. The transition probabilities between clusters are now higher and there are more long-range connections due to the faster dynamics, but the graph still comprises four clusters corresponding to the four wells of the potential. \exampleSymbol

\begin{figure}
    \centering
    \begin{minipage}[t]{0.24\textwidth}
        \centering
        \subfiguretitle{(a)}
        \vspace*{1ex}
        \includegraphics[width=0.95\textwidth]{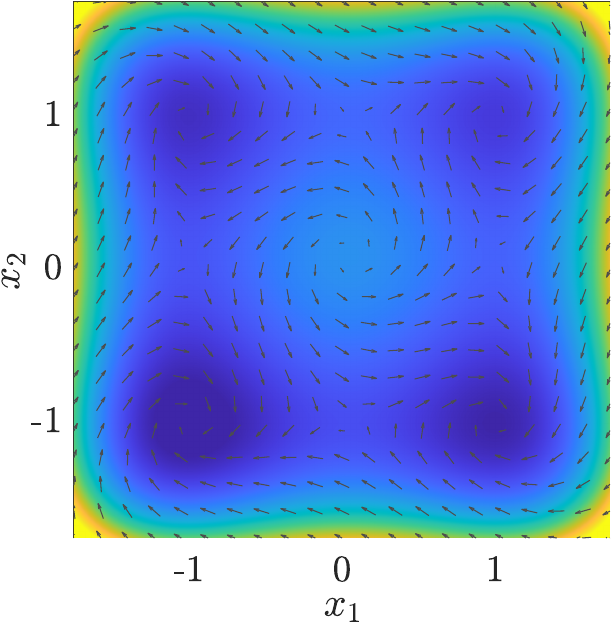}
    \end{minipage}
    \begin{minipage}[t]{0.24\textwidth}
        \centering
        \subfiguretitle{(b)}
        \vspace*{1ex}
        \includegraphics[width=0.95\textwidth]{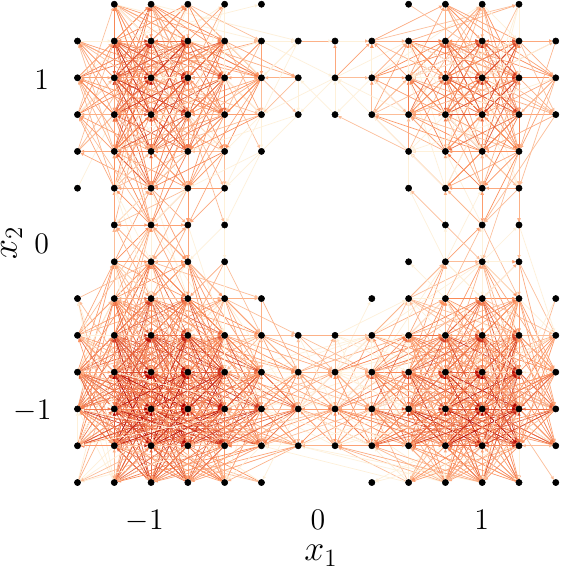}
    \end{minipage}
    \begin{minipage}[t]{0.24\textwidth}
        \centering
        \subfiguretitle{(c)}
        \vspace*{1ex}
        \includegraphics[width=0.95\textwidth]{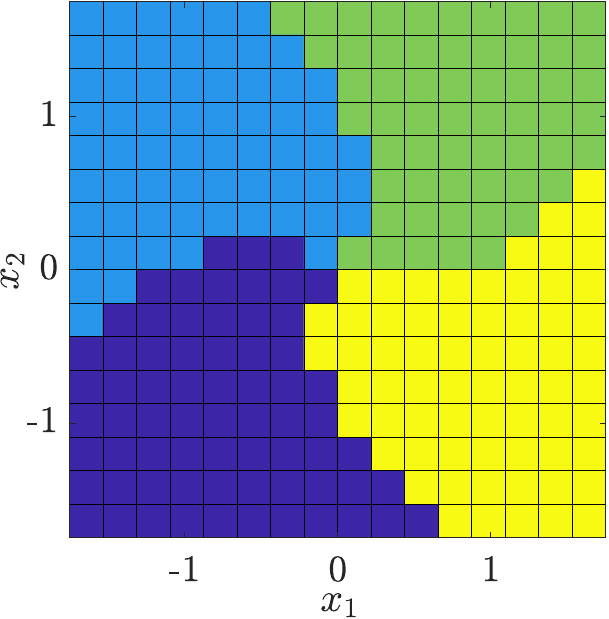}
    \end{minipage}
    \begin{minipage}[t]{0.24\textwidth}
        \centering
        \subfiguretitle{(d)}
        \vspace*{1ex}
        \includegraphics[width=0.95\textwidth]{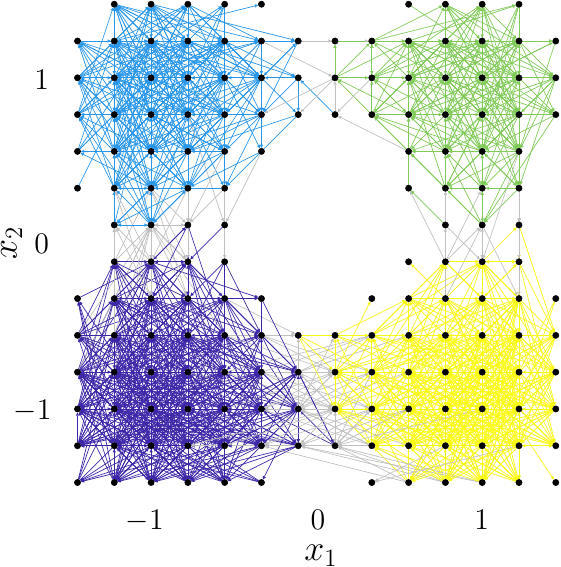}
    \end{minipage} \\[1ex]
    \begin{minipage}[t]{0.24\textwidth}
        \centering
        \subfiguretitle{(e) $ \lambda_1 = 1 $}
        \vspace*{1ex}
        \includegraphics[width=0.95\textwidth]{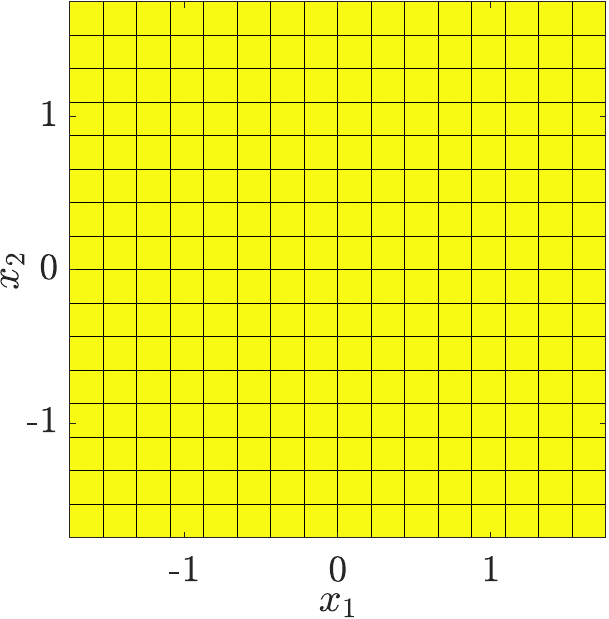}
    \end{minipage}
    \begin{minipage}[t]{0.24\textwidth}
        \centering
        \subfiguretitle{(f) $ \lambda_2 = 0.929 $}
        \vspace*{1ex}
        \includegraphics[width=0.95\textwidth]{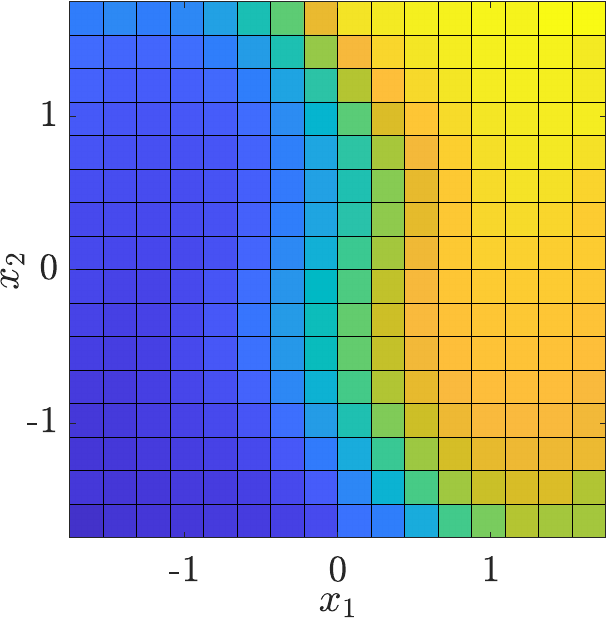}
    \end{minipage}
    \begin{minipage}[t]{0.24\textwidth}
        \centering
        \subfiguretitle{(g) $ \lambda_3 = 0.926 $}
        \vspace*{1ex}
        \includegraphics[width=0.95\textwidth]{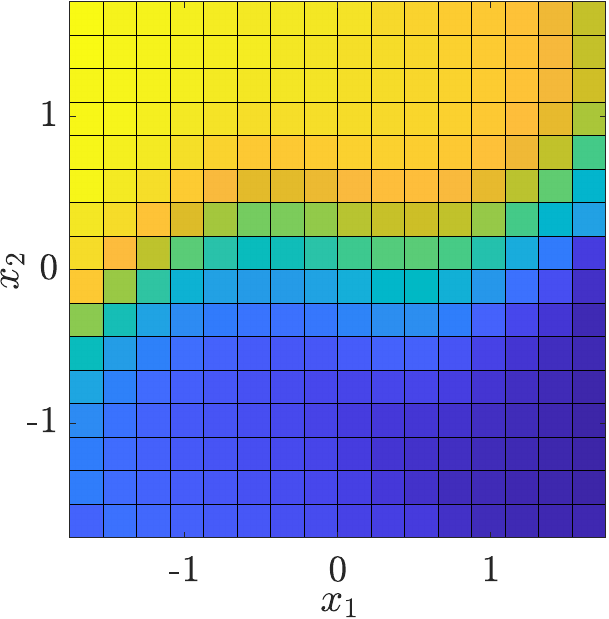}
    \end{minipage}
    \begin{minipage}[t]{0.24\textwidth}
        \centering
        \subfiguretitle{(h) $ \lambda_4 = 0.856 $}
        \vspace*{1ex}
        \includegraphics[width=0.95\textwidth]{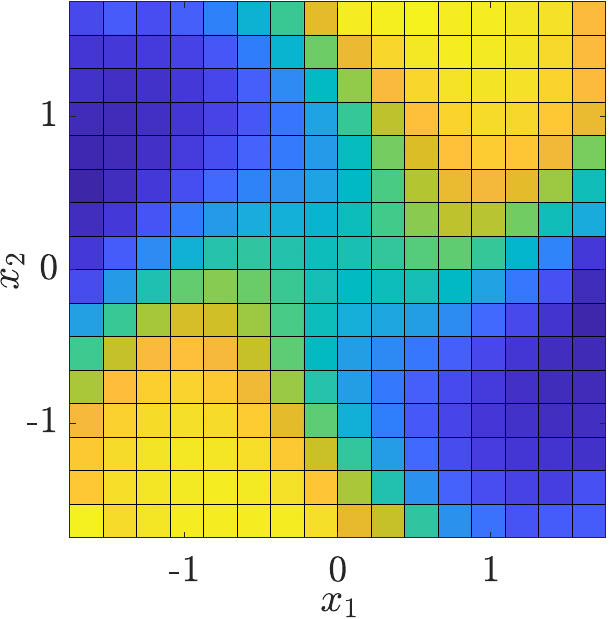}
    \end{minipage}
    \caption{(a) Quadruple-well potential and the corresponding drift term for the non-reversible process with $ c = 2 $. (b)~Resulting directed graph, where we again omit edges with low edge weights and unconnected vertices for visualization purposes. (c)~Decomposition of the domain into four coherent sets. (d)~Clustering of the graph into four weakly coupled clusters. Intra-cluster edges are again colored according to the respective clusters. (e)--(h) Dominant eigenfunctions of the forward--backward operator. Note that the eigenvalues $ \lambda_2, \dots, \lambda_4 $ are smaller than in the reversible case due to the faster dynamics.}
    \label{fig:nonreversible quadruple-well problem}
\end{figure}
\end{example}

The example illustrates that coherent sets can be regarded as a natural generalization of clusters to directed graphs. The only difference here is that we use a different transfer operator for clustering.

\subsection{Time-inhomogeneous processes and time-evolving graphs}

For the time-homogeneous systems considered above, we simply estimated the Koopman operator from one long trajectory so that the training data points are sampled from the invariant density. This is possible since the dynamics do not explicitly depend on the time $ t $. Given a time-inhomogeneous system, on the other hand, the Koopman operator is time-dependent as well.

\begin{example} \label{ex:time-dependent quadruple-well problem}
We modify the potential introduced in Example~\ref{ex:quadruple-well system} in such a way that one of the wells vanishes as shown in Figure~\ref{fig:time-dependent quadruple-well problem}. That is, the resulting stochastic differential equation and the associated transfer operators are now time-dependent. Depending on the lag time, there are either four (if $ \tau $ is small) or only three (if $ \tau $ is large) coherent sets. We choose $ \tau = 3 $, compute the forward--backward operator and its dominant eigenvalues and eigenfunctions, and then use SEBA~\cite{FRS19} to extract coherent sets. Unlike $ k $-means, SEBA computes membership functions and does not assign all boxes to one of the clusters. \exampleSymbol

\begin{figure}
    \centering
    \begin{minipage}[t]{0.24\textwidth}
        \centering
        \subfiguretitle{(a) $ t = 0 $}
        \vspace*{1ex}
        \includegraphics[width=0.95\textwidth]{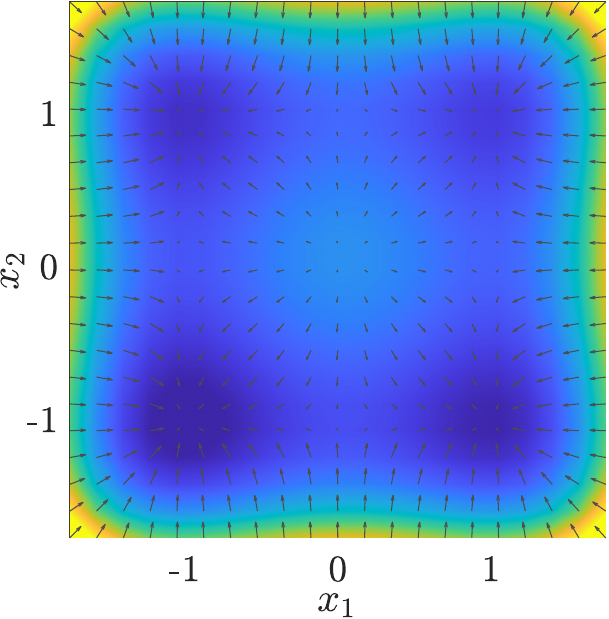} \\[1ex]
    \end{minipage}
    \begin{minipage}[t]{0.24\textwidth}
        \centering
        \subfiguretitle{(b) $ t = 1 $}
        \vspace*{1ex}
        \includegraphics[width=0.95\textwidth]{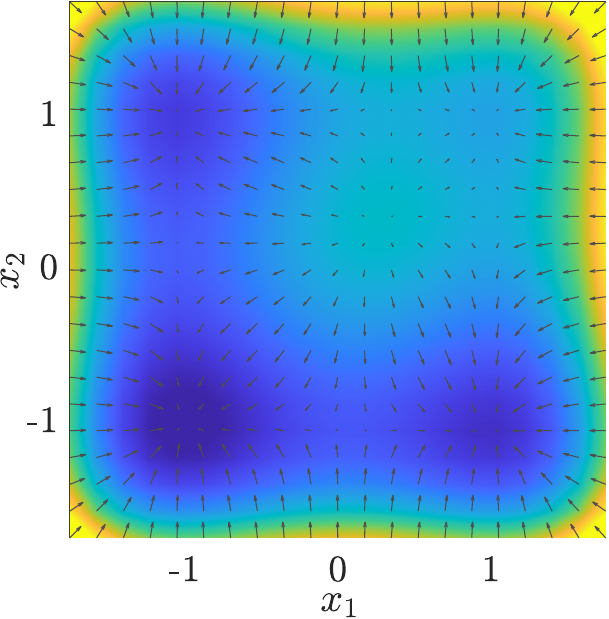} \\[1ex]
    \end{minipage}
    \begin{minipage}[t]{0.24\textwidth}
        \centering
        \subfiguretitle{(c) $ t = 2 $}
        \vspace*{1ex}
        \includegraphics[width=0.95\textwidth]{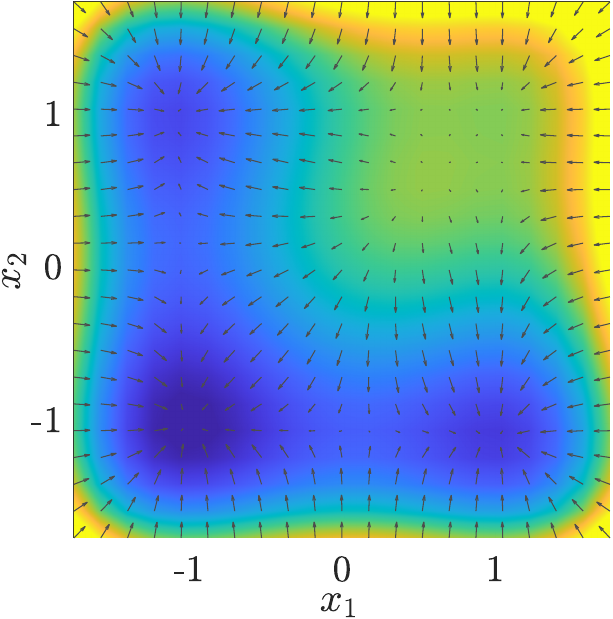} \\[1ex]
    \end{minipage}
    \begin{minipage}[t]{0.24\textwidth}
        \centering
        \subfiguretitle{(d) $ t = 3 $}
        \vspace*{1ex}
        \includegraphics[width=0.95\textwidth]{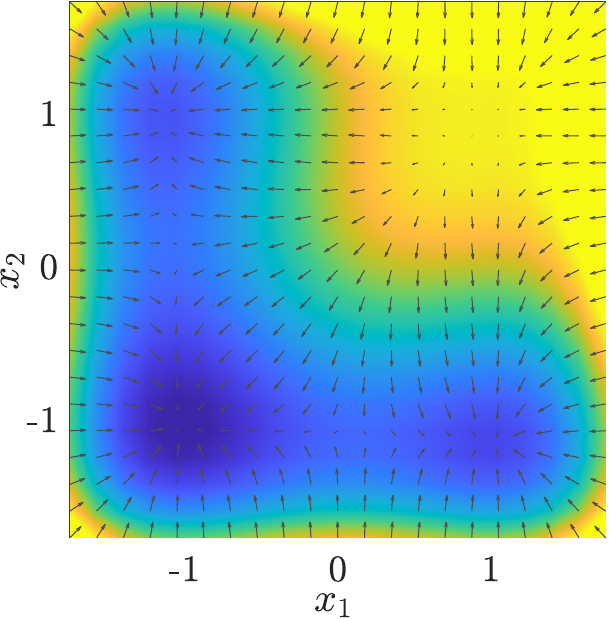} \\[1ex]
    \end{minipage} \\[1ex]
    \begin{minipage}[t]{0.24\textwidth}
        \centering
        \subfiguretitle{(e) $ t = 0 $}
        \vspace*{1ex}
        \includegraphics[width=0.95\textwidth]{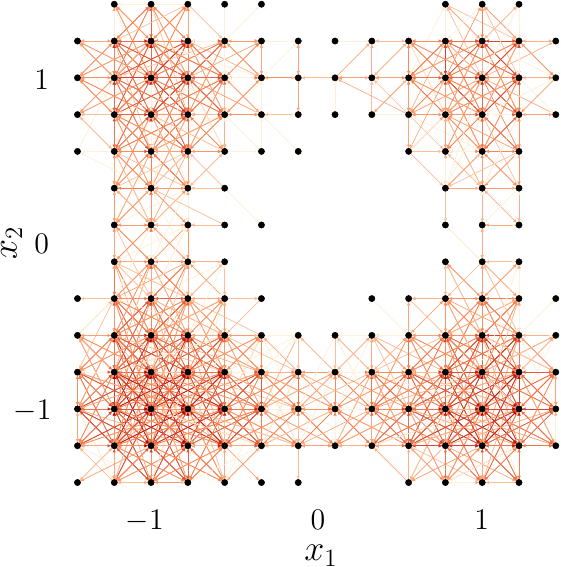}
    \end{minipage}
    \begin{minipage}[t]{0.24\textwidth}
        \centering
        \subfiguretitle{(f) $ t = 1 $}
        \vspace*{1ex}
        \includegraphics[width=0.95\textwidth]{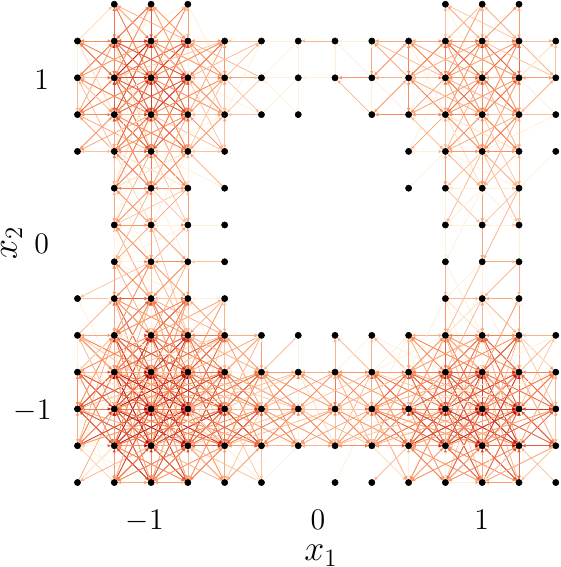}
    \end{minipage}
    \begin{minipage}[t]{0.24\textwidth}
        \centering
        \subfiguretitle{(g) $ t = 2 $}
        \vspace*{1ex}
        \includegraphics[width=0.95\textwidth]{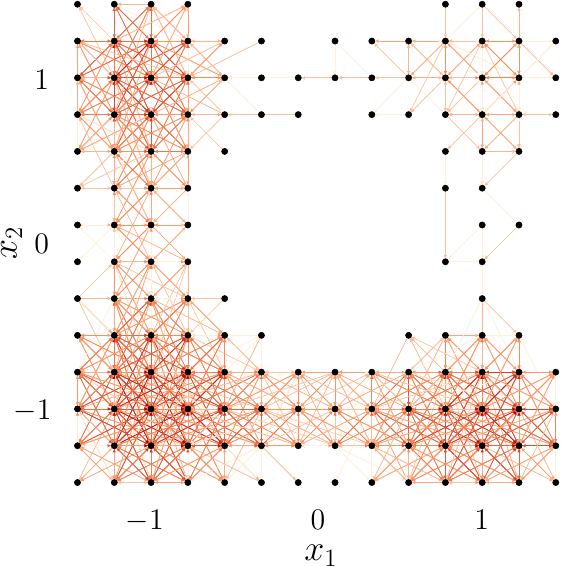}
    \end{minipage}
    \begin{minipage}[t]{0.24\textwidth}
        \centering
        \subfiguretitle{(h) $ t = 3 $}
        \vspace*{1ex}
        \includegraphics[width=0.95\textwidth]{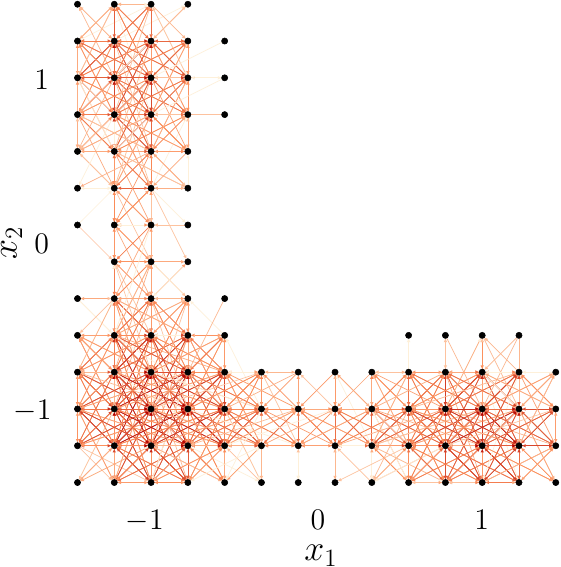}
    \end{minipage} \\[1ex]
    \begin{minipage}[t]{0.24\textwidth}
        \centering
        \subfiguretitle{(i)}
        \vspace*{1ex}
        \includegraphics[width=0.95\textwidth]{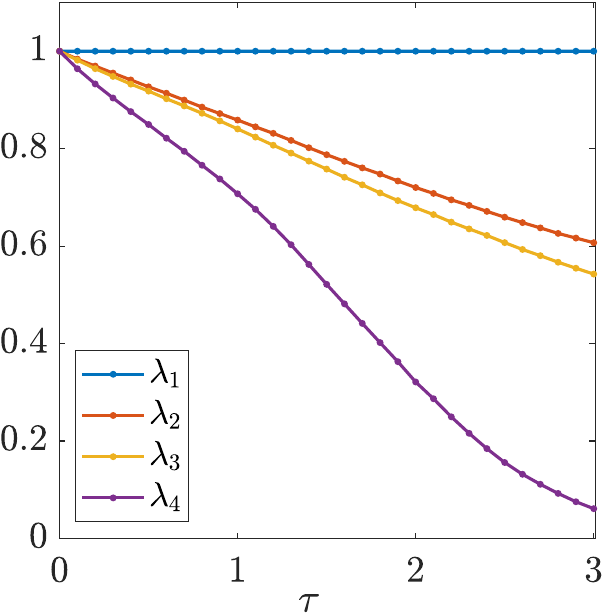}
    \end{minipage}
    \begin{minipage}[t]{0.24\textwidth}
        \centering
        \subfiguretitle{(j)}
        \vspace*{1ex}
        \includegraphics[width=0.95\textwidth]{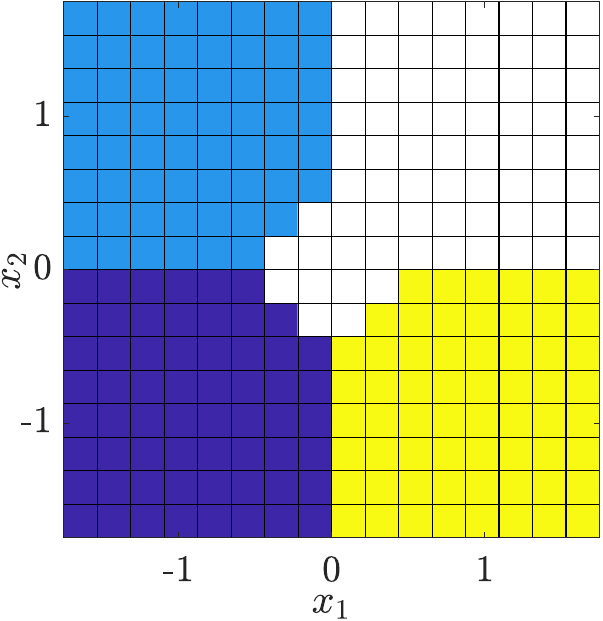}
    \end{minipage}
    \begin{minipage}[t]{0.24\textwidth}
        \centering
        \subfiguretitle{(k)}
        \vspace*{1ex}
        \includegraphics[width=0.95\textwidth]{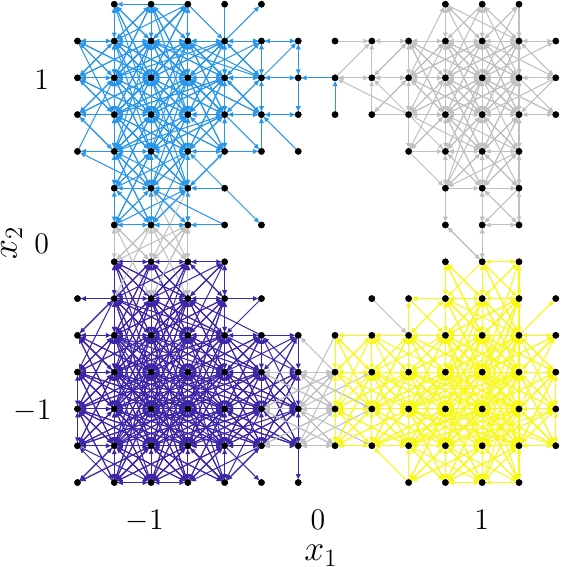}
    \end{minipage}
    \begin{minipage}[t]{0.24\textwidth}
        \centering
        \subfiguretitle{(l)}
        \vspace*{1ex}
        \includegraphics[width=0.95\textwidth]{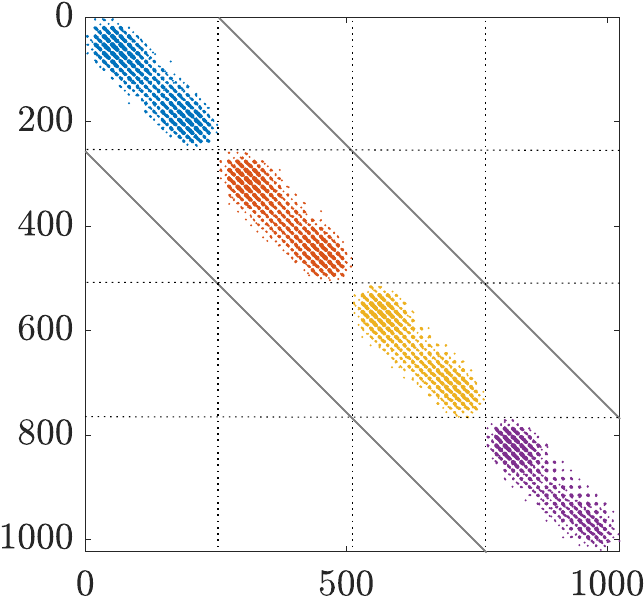}
    \end{minipage}
    \caption{(a)--(d) Time-dependent energy potential and the resulting drift terms at different times~$ t $. (e)--(h)~Corresponding graphs estimated from trajectory data. (i)~Dominant eigenvalues of $ \mathcal{F}^{\ts\tau} $ for different lag times~$ \tau $. All eigenvalues except the trivial eigenvalue $ \lambda = 1 $ decrease over time due to mixing. There are initially four dominant eigenvalues, but the fourth eigenvalue decreases quickly since one of the wells vanishes. That is, one of the coherent sets is dispersed by the dynamics. (j) Soft clustering into three coherent sets, the white boxes are not assigned to any cluster and represent the transition region. (k) Corresponding clustering of the time-evolving graph. (l) Supra-Laplacian associated with the graph. The diagonal blocks represented in different colors are the individual time layers of the graph.}
    \label{fig:time-dependent quadruple-well problem}
\end{figure}

\end{example}

The eigenfunctions of the forward--backward operator associated with a time-dependent stochastic differential equation contain important information about slowly mixing sets. The corresponding graph clustering approach is now based on random walks on a time-evolving graph. In practice, however, we would not generate random walk data to detect clusters, but directly compute the associated transfer operators using the adjacency matrices. There are different ways to take the time-dependency into account: In \cite{KD22}, we proposed computing products of transition probability matrices. This approach detects clusters that are coherent from the initial time to the final time as shown in Example~\ref{ex:time-dependent quadruple-well problem}, but fails to identify clusters that are only coherent for a short time and then vanish. Similarly, new clusters might form within the considered time interval or existing clusters might merge or split. Often several of these events will occur simultaneously so that it is difficult to detect and track such changes. In order to be able to handle these more complicated scenarios, we have to extend existing spectral clustering methods. A popular approach for clustering time-evolving networks is to ``flatten'' the graph, i.e., to create a larger static graph by connecting the different time layers in a certain way. This is illustrated in Figure~\ref{fig:time-dependent quadruple-well problem}\ts(l), where each diagonal block represents connections within a time layer and the off-diagonal blocks connections between different time layers. The eigenvalues and eigenvectors of these so-called supra-Laplacians~\cite{GDGPMA13F} depend strongly on the coupling between time layers. An approach that is more closely related to transfer operators and based on an extension of the forward--backward Laplacian was recently proposed in~\cite{TDK24}. Detecting community structures in time-evolving graphs, however, is a challenging problem and dynamical systems theory may prove beneficial in developing spectral clustering methods and also suitable benchmark problems and metrics to measure the efficacy of different methods.

\section{Conclusion}
\label{sec:Conclusion}

The main goal of this perspective article was to provide an overview of transfer operators as well as their properties and manifold applications. We have shown that it is possible to apply dynamical systems theory to graphs, but also that graph theory can be used to gain insights into characteristic properties of dynamical systems, in the hope that this will inspire further research in this interdisciplinary area. There are many interesting and challenging open problems: Since the behavior of time-evolving graphs is much more complicated---clusters can, for instance, merge and split or disappear and reappear---, a rigorous mathematical definition of clusters or community structures along with efficient and robust clustering algorithms and meaningful metrics for comparing the results are essential. The supra-Laplacian seems to be a promising approach. However, depending on the number of time layers, this can result in extremely large graphs, which may be prohibitively expensive to analyze. Furthermore, we then need to be able to distinguish between different types of clusters since some eigenvectors simply cluster the time layers of the graph into different groups but do not contain any spatial information.

Extending the approaches presented above to adaptive (co-evolving) systems and networks introduces a new set of challenges. These systems are characterized by feedback loops where changes in one component affect others and the overall behavior of the system, leading to nonlinear dynamics. These effects are often seen in real-world applications such as biological and socio-economic systems. Some examples include temporal changes of the energy landscape that are driven by environmental influences or network structure that is evolving according to the status of nodes (sick/healthy). Understanding interactions within feedback loops is essential for modeling emergent phenomena such as cluster appearance. Questions on how clusters evolve in time and if the notion of a cluster needs to be extended (e.g., to structural network clusters and status network clusters), will be a topic of future research. However, this is not easy due to the complexity and dynamic nature of these systems that are often characterized by multiple timescales, leading to complex dynamics.

Another open question is whether it is possible to extend transfer operators and the resulting clustering methods to hypergraphs, simplicial complexes, or graphons. Furthermore, given a networked dynamical system where each vertex represents a deterministic or stochastic process, the underlying graph structure will have an impact on spectral properties of associated transfer operators. Local or global symmetries, for instance, must be reflected in the eigenfunctions. Understanding these relationships might help us develop more efficient and accurate numerical methods for analyzing such interconnected systems. We also typically assume that we can observe the full state of coupled dynamical system, but this might be infeasible in practice. The question then is if we can infer global information about the system or its connectivity structure using only local or partial observations.

\section*{Acknowledgments}

We would like to thank Ginestra Bianconi for the invitation to share our perspective on the interplay between dynamical systems and complex networks.

\bibliographystyle{unsrturl}
\bibliography{DS4CN.bib}

\appendix

\section{Metastability}
\label{app:Metastability}

The term \emph{metastability} has been defined in different ways in various fields. From a mathematical point of view, several methodologies have been developed based on  hitting times \cite{Bovier2, Bovier16}, large deviation theory \cite{WF98}, and spectral methods \cite{HMSCS04}. We define metastability in terms of spectral properties of the transfer operator, as introduced in \cite{Huisinga2006, HuiDiss}. For a review of related definitions, see~\cite{Davies82a, Davies82b, Bovier16, SS13}.

We consider a full partitioning of a $ n $-dimensional state space $ \mathbb{X} $ into $ m $ sets $ \{ \mathbb{B}_1, \dots, \mathbb{B}_m \} $ such that $\bigcup_{i=1}^m \mathbb{B}_i = \mathbb{X}$ and $ \mathbb{B}_i \cap \mathbb{B}_j = \varnothing, \forall i \ne j $. The joint metastability of these sets is
\begin{equation*}
    \mathcal{D}(\mathbb{B}_1, \dots, \mathbb{B}_m) = \sum\limits_{i=1}^m p(\mathbb{B}_i, \mathbb{B}_i),
\end{equation*}
where $ p(\mathbb{B}_i, \mathbb{B}_i) = \mathbb{P}_\mu[X_T \in \mathbb{B}_i \mid X_0 \in \mathbb{B}_i]$ is the residence probability in a set $\mathbb{B}_i$. For a metastable set, the residence probability will be close to $1$ for a large timescale $T$ that is not as large as the transitions times between metastable sets. The next theorem, see also~\cite{HuiDiss, Huisinga2006}, establishes a relation between a given decomposition of a state space into metastable sets and spectral properties of the reweighted Perron--Frobenius operator $\mathcal{T}$, as introduced in Section \ref{subsec:InvariantDistr}.

\begin{theorem}
Let $ \{ \lambda_1,\ldots,\lambda_m \} $ be the $ m $ dominant eigenvalues of $\mathcal{T}$ and $ \{u_1,\ldots,u_m\} $ the corresponding eigenvectors. The joint metastability $ \mathcal{D}(\mathbb{B}_1,\ldots,\mathbb{B}_m) $ of a full partition $ \{ \mathbb{B}_1, \dots, \mathbb{B}_m \} $ of the state space can be bounded from below and above by
\begin{equation*}
    1 + \delta_2\lambda_2 + \dots + \delta_m\lambda_m + c \leq \mathcal{D}(\mathbb{B}_1,\dots,\mathbb{B}_m) \leq 1 +\lambda_2 + \dots + \lambda_m,
\end{equation*}
where $ c = \lambda_{m+1}( 1-\delta_2) \cdots (1-\delta_m) $ and $\delta_j = \|Q u_j\|_{L_\mu^2}^2, j = 2, \dots, m $, is the error of the orthogonal projection $Q$ with respect to the $\mu$-weighted scalar product of the eigenvector $u_j$ onto the space spanned by the characteristic functions $ D = \mspan\{\mathds{1}_{\mathbb{B}_1}, \dots, \mathds{1}_{\mathbb{B}_m}\} $.
\end{theorem}
This theorem states that the metastability of a decomposition into $m$ sets cannot be larger than the sum of the $ m $ dominant eigenvalues of the transfer operator. The maximal metastability is achieved when $ \delta_j \approx 1 $, for $ j = 2, \dots, m$ , i.e., when the dominant eigenvectors are almost constant on the metastable sets. Thus, by minimizing the projection error $\delta_j$ we can find the optimal  partition of the state space into metastable sets. This is the idea behind many spectral clustering methods.

\section{Properties of transfer operators}
\label{app:Properties of transfer operators}

The following properties can be used to show that the spectra of the forward--backward and backward--forward operators are real-valued.

\begin{lemma} \label{lem:adjointness}
It holds that $ \innerprod{\mathcal{T} u}{f}_\nu = \innerprod{u}{\mathcal{K} f}_\mu $.
\end{lemma}
\begin{proof}
We have
\begin{align*}
    \innerprod{\mathcal{T}^{\ts\tau} u}{f}_\nu
        &= \intop_\mathbb{X} \frac{1}{\nu(x)} \!\intop_\mathbb{X}\! p_\tau(y, x) \ts \mu(y) \ts u(y) \ts \mathrm{d} y \ts f(x) \ts \nu(x) \ts \mathrm{d}x \\
        &= \intop_\mathbb{X} u(y) \!\intop_\mathbb{X}\! p_\tau(y, x) \ts f(x) \ts \mathrm{d}x \ts \mu(y) \ts \mathrm{d} y \\
        &= \innerprod{u}{\mathcal{K}^{\ts\tau} f}_\mu.
\end{align*}
For the graph case, this can be shown in an analogous fashion, see \cite{KT24}.
\end{proof}

\begin{corollary} \label{cor:self-adjoint and positive definite}
It follows that $ \mathcal{F} $ is self-adjoint w.r.t.\ the $ \mu $-weighted inner product and $ \mathcal{B} $ w.r.t.\ the $ \nu $-weighted inner product. Furthermore, $ \mathcal{F} $ and $ \mathcal{B} $ are positive semi-definite.
\end{corollary}

\begin{proof}
Using the definition of $ \mathcal{F} $, we obtain
\begin{align*}
    \innerprod{\mathcal{F} u_1}{u_2}_\mu
        &= \innerprod{\mathcal{K} \mathcal{T} u_1}{u_2}_\mu
        = \innerprod{\mathcal{T} u_1}{\mathcal{T} u_2}_\nu
        = \innerprod{u_1}{\mathcal{K} \mathcal{T} u_2}_\mu
        = \innerprod{u_1}{\mathcal{F} u_2}_\mu.
\end{align*}
In particular, $ \innerprod{\mathcal{F} u}{u}_\mu = \norm{\mathcal{T} u}_\nu^2 \ge 0 $. The results for the backward--forward operator follow in the same way.
\end{proof}

\begin{lemma} \label{lem:spectral radius}
The spectral radii of $ \mathcal{F} $ and $ \mathcal{B} $ are $ 1 $.
\end{lemma}

\begin{proof}
We show only the proof for $ \mathcal{F}^{\ts\tau} $. Let $ \mathds{1} $ denote the function that is one everywhere, then $ \mathcal{K}^{\ts\tau} \mathds{1} = \mathds{1} $ since $ p_\tau(x, \cdot) $ is a probability density. Similarly, $ \mathcal{T}^{\ts\tau} \mathds{1} = \frac{1}{\nu} \mathcal{P}^{\ts\tau} \mu = \mathds{1} $ by construction. It was shown in \cite{Froyland13} that then $ \norm{\mathcal{T}^{\ts\tau}} = \norm{\mathcal{K}^{\ts\tau}} \le 1 $ and
\begin{equation*}
    \norm{\mathcal{F}^{\ts\tau}} = \norm{\mathcal{K}^{\ts\tau} \mathcal{T}^{\ts\tau}} \le \norm{\mathcal{K}^{\ts\tau}} \norm{\mathcal{T}^{\ts\tau}} \le 1,
\end{equation*}
but $ \mathcal{F}^{\ts\tau} \mathds{1} = \mathcal{K}^{\ts\tau} \mathcal{T}^{\ts\tau} \mathds{1} = \mathds{1} $ so that $ \norm{\mathcal{F}^{\ts\tau}} = 1 $. Thus, also the spectral radius is $ 1 $. The proof for the graph case can be found in~\cite{KT24}.
\end{proof}

\end{document}